\documentclass[a4paper]{article}

\usepackage[T1]{fontenc}
\usepackage{lmodern}
\usepackage[a4paper,margin=30mm]{geometry}
\usepackage{amsmath,amssymb,mathtools,amsthm}
\usepackage{microtype}
\usepackage{array}
\usepackage{xcolor}
\usepackage{tikz}
\usepackage[hidelinks]{hyperref}

\hypersetup{
  pdftitle={Optimal Actuator Design across Ranks and Control Horizons},
  pdfauthor={Emmanuel Trelat and Enrique Zuazua},
  pdfsubject={Rank-constrained actuator design and controllability Gramians},
  pdfkeywords={actuator design, controllability Gramian, rank constraint, phase alignment}}

\numberwithin{equation}{section}
\numberwithin{figure}{section}
\numberwithin{table}{section}

\theoremstyle{plain}
\newtheorem{theorem}{Theorem}[section]
\newtheorem{lemma}[theorem]{Lemma}
\newtheorem{proposition}[theorem]{Proposition}
\newtheorem{corollary}[theorem]{Corollary}

\theoremstyle{definition}
\newtheorem{example}[theorem]{Example}
\theoremstyle{remark}
\newtheorem{remark}[theorem]{Remark}

\def\R{\mathbb{R}}
\def\N{\mathbb{N}}
\newcommand{\C}{\mathbb{C}}

\renewcommand{\geq}{\geqslant}
\renewcommand{\leq}{\leqslant}

\newcommand{\Id}{\mathrm{Id}}
\newcommand{\tr}{\operatorname{tr}}
\newcommand{\rank}{\operatorname{rank}}
\newcommand{\Span}{\operatorname{span}}
\newcommand{\diag}{\operatorname{diag}}
\newcommand{\dist}{\operatorname{dist}}
\newcommand{\calG}{\mathcal{G}}

\title{Optimal Actuator Design across Ranks and Control Horizons}

\author{
Emmanuel Tr\'elat\thanks{Sorbonne Universit\'e, Universit\'e Paris Cit\'e, CNRS, Inria, Laboratoire Jacques-Louis Lions (LJLL), F-75005 Paris, France (\texttt{emmanuel.trelat@sorbonne-universite.fr}).}
\and
Enrique Zuazua\thanks{Chair for Dynamics, Control, Machine Learning and Numerics (AvH Professorship), Department of Mathematics, FAU Erlangen-N\"urnberg, 91058 Erlangen, Germany; Chair of Computational Mathematics, Fundaci\'on Deusto, Bilbao, Spain; Departamento de Matem\'aticas, Universidad Aut\'onoma de Madrid, Spain (\texttt{enrique.zuazua@fau.de}).}
}

\date{}

\begin{document}

\maketitle

\begin{abstract}
We study how actuator rank and control horizon jointly determine worst-case control performance for finite-dimensional linear systems under a fixed total actuator-gain budget. The design variable is the input covariance $X=BB^\top$. Scalar actuators give rank-one matrices $X=bb^\top$; dropping the rank constraint yields their convex hull and a semidefinite benchmark. We identify regimes in which low-rank designs necessarily fall short and others in which they attain the benchmark. At small time, the relaxed maximizer is unique and full rank. If $A$ is cyclic, then, for every $k<n$, optimal rank-$k$ performance becomes negligible relative to the relaxed optimum as the horizon tends to zero, with a precise algebraic rate. At rank one, we compute the leading coefficient and next-order term; for real normal matrices with simple complex spectrum, we determine the leading spectral weights and their first horizon-dependent correction. For symmetric matrices, the full-rank obstruction persists at every positive time. For skew-symmetric matrices, by contrast, phase-aligned horizons can make the time-averaged covariance of a low-rank actuator exactly isotropic, closing the relaxation gap. We characterize scalar recovery, determine the minimum recovery rank at every horizon, and distinguish exact or near recovery from ordinary long-time averaging. Galerkin truncations of the heat, Schr\"odinger, and wave equations illustrate the contrast between concentrated small-time profiles and flat profiles at phase-aligned horizons. These PDE statements concern finite-dimensional Galerkin optimizers and their renormalized modal limits; they do not assert existence of an optimizer for the full PDE.
\end{abstract}


\bigskip


\section{Introduction}

We study the maximization of the smallest eigenvalue of the finite-horizon controllability Gramian over input covariance matrices with unit trace and a prescribed rank bound. Dropping the rank constraint yields a semidefinite relaxation that provides a benchmark for the constrained problem. Our aim is to determine how optimal performance depends jointly on the actuator rank and the control horizon, and when designs of a prescribed rank can attain this benchmark.

Let $A\in\R^{n\times n}$ and $T>0$. A unit vector $b\in\R^n$ defines the scalar-input system
$$
\dot x(t)=Ax(t)+bu(t), \qquad u(t)\in\R,
$$
and its controllability Gramian
$$
G_T(A,b)=\int_0^T e^{tA}bb^\top e^{tA^\top}\,dt.
$$
When $(A,b)$ is controllable, the largest $L^2$ norm of the minimum-energy control needed to steer the system from the origin to a unit target at time $T$ is
$$
C_T(A,b)=\frac{1}{\sqrt{\lambda_{\min}(G_T(A,b))}}.
$$
We set $C_T(A,b)=+\infty$ otherwise. Minimizing this worst-case control cost is therefore equivalent to solving
\begin{equation}\label{gamma1_intro}
\gamma_1(A,T)
=
\max_{\Vert b\Vert=1}
\lambda_{\min}(G_T(A,b)).
\end{equation}

To compare scalar and multi-input designs under the same total actuator-gain budget, consider
$$
\dot x(t)=Ax(t)+Bu(t),
\qquad
B\in\R^{n\times m},
\qquad
u(t)\in\R^m,
$$
with $\tr(B^\top B)=1$. The Gramian
$$
G_T(A,B)=\int_0^T e^{tA}BB^\top e^{tA^\top}\,dt
$$
depends on $B$ only through the covariance $X=BB^\top$. This matrix is positive semidefinite, has trace one, and its rank is the least number of independent input directions needed to realize it. Scalar actuators correspond to rank-one covariances $X=bb^\top$.

Every positive semidefinite matrix of trace one is a convex combination of rank-one covariances of trace one. Dropping the rank constraint thus gives the convex hull of the scalar admissible set and the relaxed problem
\begin{equation}\label{gamma_rel_intro}
\gamma_{\mathrm{rel}}(A,T)
=
\max_{X\succeq0,\ \tr X=1}
\lambda_{\min}\bigg(
\int_0^T e^{tA}Xe^{tA^\top}\,dt
\bigg).
\end{equation}
This semidefinite relaxation provides an upper benchmark for designs with a limited number of input directions. More generally, we denote by $\gamma_k(A,T)$ the optimal value among covariances of rank at most $k$; see \eqref{gammam}. Thus $\gamma_1(A,T)\leq\gamma_2(A,T)\leq\cdots\leq\gamma_n(A,T)=\gamma_{\mathrm{rel}}(A,T)$.

Our central questions concern attainment of this benchmark and the performance loss when it cannot be attained. At rank $k$, exact recovery means
$\gamma_k(A,T)=\gamma_{\mathrm{rel}}(A,T)$.
Since the admissible covariance sets are compact, this equality is equivalent to the existence of a relaxed maximizer of rank at most $k$. The least rank among all relaxed maximizers therefore gives the minimum number of input directions required for exact recovery. In the scalar case, the ratio $\rho_1(A,T)=\gamma_1(A,T)/\gamma_{\mathrm{rel}}(A,T)$ measures the fraction of the relaxed benchmark retained by the optimal single input direction, while $1-\rho_1(A,T)$ is the relative performance gap.

An explicit illustration is provided by the planar rotation
$$
A=\omega J,
\qquad
J=\begin{pmatrix}0&-1\\1&0\end{pmatrix},
\qquad
\omega\neq0,
$$
analyzed in Example~\ref{ex_planar_rotation} and plotted in Figure~\ref{fig_planar_ratio}. This particular system exhibits small-time loss, exact recovery at phase-aligned horizons, and asymptotic recovery through long-time averaging.

\begin{figure}[ht]
\centering
\begin{tikzpicture}[
  x=0.95cm,
  y=3.6cm,
  line width=1pt,
  every node/.style={inner sep=2pt}
]
\draw[->] (0,0) edge (10.9,0);
\node[right] at (10.9,0) {$\vert\omega\vert T$};
\draw[->] (0,0) edge (0,1.30);
\node[above] at (0,1.30) {$\rho_1$};
\node[below left] at (0,0) {$0$};

\draw (-0.1,1) edge (0.1,1);
\node[left] at (-0.1,1) {$1$};
\draw[densely dashed,gray] (0,1) edge (10.5,1);

\foreach \x/\lab in {3.1416/\pi,6.2832/2\pi,9.4248/3\pi}{
  \draw[gray!45,densely dotted] (\x,0) edge (\x,1);
  \node[below] at (\x,0) {$\lab$};
}

\draw[
  line width=1pt,
  domain=0.05:10.5,
  samples=260,
  smooth
]
plot (\x,{1-abs(sin(\x r))/\x});

\foreach \x in {3.1416,6.2832,9.4248}{
  \fill (\x,1) circle (1.3pt);
}

\node[fill=white] at (5,1.17) {phase-aligned returns};
\node[align=left,anchor=west,fill=white] at (1.6,0.17)
  {quadratic decay\\of the recovery ratio};
\draw[->,gray] (1.52,0.17) edge (1,0.152);
\node[align=right,anchor=east,fill=white] at (10.4,0.45)
  {long-time\\averaging};
\end{tikzpicture}
\caption{Scalar recovery ratio for the planar rotation
$A=\omega J\in\R^{2\times2}$, with $\omega\neq0$:
$\rho_1(A,T)=1-\vert\sin(\omega T)\vert/(\vert\omega\vert T)$.
As $T\to0^+$,
$\rho_1(A,T)=(\omega T)^2/6+\mathrm{O}(T^4)$,
in agreement with \eqref{intro_short_ratio} for $n=2$.
Exact recovery occurs at $\vert\omega\vert T\in\pi\N^*$,
and $\rho_1(A,T)\to1$ as $T\to+\infty$.
This figure concerns the planar example; the finite-time returns
and the long-time limit shown here are not properties of general
matrices $A$.}
\label{fig_planar_ratio}
\end{figure}

\medskip
{\bf Contributions of this paper.}
For every matrix $A$, the relaxed maximizer is unique and positive definite for all sufficiently small $T>0$. Consequently,
$\gamma_k(A,T)<\gamma_{\mathrm{rel}}(A,T)$
for every $k<n$ at such horizons. This obstruction is independent of any cyclicity assumption on $A$.

If $A$ is cyclic, meaning that
$\Span(b,Ab,\ldots,A^{n-1}b)=\R^n$
for some vector $b$, we obtain the sharper comparison
$$
\frac{\gamma_k(A,T)}{\gamma_{\mathrm{rel}}(A,T)}
\sim
n c_k(A)T^{2\lceil n/k\rceil-2},
\qquad T\to0^+,
\qquad 1\leq k<n,
$$
with $c_k(A)>0$. Thus even the optimal design with at most $k<n$ input directions retains a vanishing fraction of the relaxed benchmark. The exponent is twice the least achievable maximal Kalman depth with at most $k$ input directions. At intermediate ranks, the leading constant exists and admits a lower bound in terms of fixed-covariance coefficients; whether this bound is exact remains open.

At rank one, we identify the leading coefficient explicitly:
\begin{equation}\label{intro_short_ratio}
\frac{\gamma_1(A,T)}{\gamma_{\mathrm{rel}}(A,T)}
=
n\,\kappa_n\,\delta_n(A)^2
T^{2n-2}(1+\mathrm{o}(1)),
\qquad
\kappa_n
=
\frac{((n-1)!)^2}{(2n-1)!(2n-2)!},
\end{equation}
where $\delta_n(A)$ is the largest possible last transverse height of a Krylov chain generated by a unit vector. We also compute the first relative correction to \eqref{intro_short_ratio}. For real normal matrices with simple complex spectrum, we determine the unique leading spectral weights and the first variation of the exact maximizing weights with $T$. These exact weights are generally horizon dependent. For such normal matrices and all sufficiently small positive horizons, the maximizing actuator is unique modulo the orthogonal symmetries commuting with $A$, although it is never literally unique. No uniqueness modulo these symmetries is asserted for a general cyclic matrix.

\smallskip
For symmetric matrices, the small-time obstruction persists at every positive horizon: the relaxed maximizer is unique and positive definite for every $T>0$. Hence
$\gamma_k(A,T)<\gamma_{\mathrm{rel}}(A,T)$
for every $k<n$ and every $T>0$.

\smallskip
For skew-symmetric matrices, exact low-rank recovery is possible at special horizons. Every normalized Gramian has trace $T$, and the covariance $X=\Id/n$ attains the resulting upper bound, so
$\gamma_{\mathrm{rel}}(A,T)=T/n$.
The scalar gap closes if and only if $A$ is cyclic and $e^{TA}=\sigma\,\Id$ for some $\sigma\in\{-1,1\}$.
If $\ker A\neq\{0\}$, necessarily $\sigma=1$.
At such a horizon, a suitably balanced scalar actuator has the isotropic Gramian $T\,\Id/n$.

More generally, we group the signed eigenfrequencies, including zero when present, according to their endpoint phases. We refer to coincidences of these phases as phase alignment, and use resonance for the corresponding horizons. The least rank attaining $T/n$ is obtained by taking the largest complex spectral multiplicity in each endpoint phase class and summing over the distinct classes. This yields full rank at generic horizons, intermediate ranks at partial resonances, and rank one at complete phase alignment when $A$ is cyclic.

\smallskip
Exact recovery, near recovery, and long-time averaging must be distinguished. For a fixed skew-symmetric matrix $A$, as the scalar ratio $\rho_1(A,T)$ approaches one over bounded horizons, the endpoint flow approaches a common sign on all nonzero frequency subspaces. The quantitative estimate depends on $\Vert A\Vert T$. In contrast, for every fixed cyclic skew-symmetric matrix,
$\rho_1(A,T)\to1$ as $T\to+\infty$,
without requiring the endpoint flow to approach a common return. Indeed, for a balanced actuator, the off-diagonal correlations in a complex eigenbasis remain bounded, while the diagonal entries grow linearly with $T$.

\smallskip
We finally apply these results to spectral truncations of distributed systems. For the first $N$ modes of the heat equation, the leading small-time weights concentrate at frequencies of order $\sqrt N$ as $N\to+\infty$ and reveal, after renormalization, a dipole-type spectral envelope. Complex Schr\"odinger truncations have the same leading small-time weights, although their finite-horizon corrections differ; at revival times, their maximizing weights are exactly flat. At common periods, the wave equation in energy coordinates also has flat maximizing modal weights, even when the force is constrained to act on the velocity component. These are finite-dimensional results followed by modal limits. No maximizer for the full partial differential equation is claimed without an admissible actuator class and the required uniform estimates; a small-time PDE limit would additionally require a controlled joint regime in the truncation dimension and the horizon.

\medskip
{\bf Comments on the existing literature.}
For a fixed controllable pair, the small-time Gramian scale and its relation to the Kalman filtration are classical. The spectral expansion used below is established in the companion paper~\cite{TrelatZuazua2026gramian}, and its leading fast-control scale is consistent with \cite{Seidman1988,SeidmanYong1997}. The present paper passes from fixed-pair spectral analysis to optimization over normalized covariances with a prescribed rank bound, and compares the resulting optimal values with the semidefinite benchmark.

Worst-case energy actuator design for symmetric matrices was considered in \cite{ChenBelabbas2017}, while the scalar finite-dimensional problem and a method based on the Brunovsk\'y coordinates of \cite{Brunovsky1970} were developed in \cite{GeshkovskiZuazua2022}. Optimization over an input covariance under a power constraint appears, for stable systems and the infinite-horizon Gramian, in \cite{BaggioZampieriScherer2019}. We therefore do not claim the covariance relaxation itself as new. Our contribution is its horizon-dependent comparison with rank-constrained designs: the sharp rank-dependent small-time equivalents, the two-term scalar recovery ratio, the displacement of normal maximizing weights, the exact recovery rank at every skew-symmetric horizon, the quantitative distinction between near phase alignment and long-time averaging, and the consequences for modal truncations. Appendix~\ref{sec_brunovsky} also identifies the error by which the separated Brunovsk\'y objective in \cite{GeshkovskiZuazua2022} ceases to represent the original worst-energy criterion, as recorded in the corrigendum~\cite{GeshkovskiZuazuaCorrigendum}.

Gramian metrics, actuator and sensor selection, and convex design have a broad literature; see \cite{JoshiBoyd2009,Morris2011,Olshevsky2014,PasqualettiZampieriBullo2014,SummersCortesiLygeros2016}. Related design problems for distributed systems appear in \cite{PrivatTrelatZuazua2015,PrivatTrelatZuazua2016,PrivatTrelatZuazua2017}.

For a controllable pair, the orbit $t\mapsto e^{tA}b$ defines a continuous frame with Gramian $G_T$ as its frame operator. Tightness means $G_T=c\,\Id$ for some $c>0$; in the skew-symmetric setting with $\Vert b\Vert=1$, this becomes $G_T=T\,\Id/n$. Relevant finite-dimensional and group-orbit viewpoints can be found in \cite{AceskaKim2017,AldroubiHuangPetrosyan2019,AssifSheriffChatterjee2021,DiazMartinMedriMolter2021,Iverson2018}; see also \cite{AguileraCabrelliNegreiraPaternostro2025,AldroubiCabrelliKrishtalMolter2026}. The complete-period multiplicity rule follows from compact-group orbit-frame theory. The sum over endpoint phase classes on an arbitrary finite window is the additional feature needed here.

\medskip
{\bf Structure of the paper.}
Section~\ref{sec_obstructions} formulates the relaxation and proves the small-time and symmetric obstructions to exact low-rank recovery, together with the sharp small-time comparisons. Section~\ref{sec_skew} characterizes exact recovery for skew-symmetric dynamics. Section~\ref{sec_near} separates quantitative phase alignment from long-time averaging. Section~\ref{sec_PDE} derives the modal consequences for the heat, Schr\"odinger, and wave equations. The appendix contains the real phase-block construction and the correction of the Brunovsk\'y reduction. Matrix norms are operator norms throughout, unless another norm is specified.


\section{Covariance relaxation and rank obstructions}\label{sec_obstructions}
The scalar actuator problem remains our primary design problem. The covariance relaxation is its convexification and provides an upper benchmark. We now quantify both the value gap created by the rank-one constraint and the least rank at which this gap disappears. The first obstruction below is universal but local in time; the second is specific to symmetric dynamics and persists at every positive horizon.
\subsection{Multi-input formulation and recovery rank}
For a symmetric matrix $X$, set
$$
\calG_T(X) = \int_0^T e^{tA}Xe^{tA^\top}dt.
$$
For $k\in\{1,\ldots,n\}$, define
\begin{equation}\label{gammam}
\gamma_k(A,T) = \max_{\substack{X\succeq0,\ \tr X = 1\\ \rank X \leq k}}\lambda_{\min}(\calG_T(X)).
\end{equation}
Every feasible covariance in \eqref{gammam} factors as $X = BB^\top$ with $B\in\R^{n\times m}$, where $m = \rank X \leq k$, and $\tr(B^\top B) = 1$. Conversely, every such input matrix defines an admissible covariance. Hence \eqref{gammam} is precisely the best normalized design with at most $k$ input directions, and
$$
\gamma_1(A,T) \leq \gamma_2(A,T) \leq \cdots \leq \gamma_n(A,T) = \gamma_{\mathrm{rel}}(A,T).
$$
All maxima are attained, since the admissible covariance sets are compact. At rank $k$, the relaxation has no gap precisely when $\gamma_k(A,T) = \gamma_{\mathrm{rel}}(A,T)$.

Equivalently, $\gamma_{\mathrm{rel}}(A,T)$ is the value of the semidefinite program $\max\{\gamma\mid \calG_T(X)\succeq\gamma\Id,\ X\succeq0,\ \tr X = 1\}$. Since every feasible $X$ is a mixture of rank-one covariances, this has exactly the form of an approximate $E$-optimal design problem for the family of Gramians $\{G_T(A,b)\mid\Vert b\Vert = 1\}$, in the sense of \cite{Pukelsheim2006}; see also \cite{JoshiBoyd2009,BaggioZampieriScherer2019}.

We shall use two complementary diagnostics. The scalar recovery ratio is
$$
\rho_1(A,T) = \frac{\gamma_1(A,T)}{\gamma_{\mathrm{rel}}(A,T)}\in[0,1],
$$
and the exact recovery rank is
\begin{equation}\label{recovery_rank}
r(A,T) = \min\{\rank X\ \mid\ X\succeq0,\ \tr X = 1,\ \lambda_{\min}(\calG_T(X)) = \gamma_{\mathrm{rel}}(A,T)\}.
\end{equation}
Compactness of the rank strata gives
\begin{equation}\label{rank_equivalence}
\gamma_k(A,T) = \gamma_{\mathrm{rel}}(A,T) \ \Leftrightarrow\ k \geq r(A,T).
\end{equation}
Thus $1 - \rho_1(A,T)$ measures the relative scalar value gap, while $r(A,T)$ counts the input directions required to close the gap exactly.

The following elementary certificate will be useful below. The adjoint of a linear map on symmetric matrices is taken for the trace inner product.

\begin{lemma}[inverse certificate]\label{lem_inverse_certificate}
Let $\mathcal L:\operatorname{Sym}_n\to\operatorname{Sym}_n$ be invertible. Assume that $H = \mathcal L^{-1}(\Id)\succ0$ and $K = (\mathcal L^*)^{-1}(\Id)\succ0$.
Then $\tr H = \tr K = \tau$, and
$$
\max_{X\succeq0,\ \tr X = 1}\lambda_{\min}(\mathcal L(X)) = \frac{1}{\tau}.
$$
The unique maximizer is $X_\star = H/\tau$.
\end{lemma}

\begin{proof}
Adjointness gives $\tr H = \langle \Id,\mathcal L^{-1}(\Id) \rangle = \langle (\mathcal L^*)^{-1}(\Id),\Id \rangle = \tr K$.
The candidate $X_\star$ is feasible and $\mathcal L(X_\star) = \Id/\tau$. Conversely, if $X$ is feasible and $\mu = \lambda_{\min}(\mathcal L(X))$, then $1 = \tr(K\mathcal L(X)) \geq \mu\tr K = \mu\tau$.
Equality forces $\tr(K(\mathcal L(X) - \Id/\tau)) = 0$. Since both factors are positive semidefinite and $K$ is positive definite, $\mathcal L(X) = \Id/\tau$, and invertibility gives $X = X_\star$.
\end{proof}

\subsection{The universal small-time obstruction}

Unlike the scalar problem, the relaxed problem has a nonsingular first-order limit as $T\to 0^+$: $\calG_T(X)/T\to X$ uniformly on bounded sets of symmetric matrices. The next theorem identifies the covariance selected by the first perturbation of this limit and gives one further order in both the optimal value and the optimizer.

\begin{theorem}[small-time relaxed design]\label{thm_relaxed_jet}
Let $S = (A + A^\top)/2$. Set $c = \tr A/n$ and $R_A = A^2 + 2AA^\top + (A^\top)^2$. As $T\to0^+$,
\begin{equation}\label{relaxed_jet}
\gamma_{\mathrm{rel}}(A,T) = \frac{T}{n} + \frac{\tr A}{n^2}T^2 + \left(\frac{(\tr A)^2}{n^3} - \frac{\tr R_A}{12n^2}\right)T^3 + \mathrm{O}(T^4).
\end{equation}
The relaxed maximizer $X_T$ is unique for every sufficiently small $T>0$, depends analytically on $T$ near zero, and satisfies
\begin{equation}\label{relaxed_optimizer}
X_T = \frac{1}{n}\bigg(\Id + T(c\Id - S) + T^2\Big(\Big(c^2 - \frac{\tr R_A}{12n}\Big)\Id - cS + \frac{1}{12}R_A\Big)\bigg) + \mathrm{O}(T^3),
\end{equation}
as $T\to 0^+$. In particular, $r(A,T) = n$ for every sufficiently small $T>0$.
\end{theorem}

Here uniqueness concerns the covariance $X_T$. Its realizations as input matrices through factorizations $X_T = B_TB_T^\top$ need not be unique.

\begin{proof}
As operators on $\operatorname{Sym}_n$, one has $\calG_T = T(\mathcal I + \frac{T}{2}\mathcal A + T^2\mathcal E_T)$, where $\mathcal I$ is the identity, $\mathcal A(X) = AX + XA^\top$, and $T\mapsto\mathcal E_T$ is analytic and bounded near $T = 0$.
Hence $\calG_T$ is invertible for small $T$. More precisely, the operator power series $z/(e^z - 1) = 1 - z/2 + z^2/12 + \mathrm{O}(z^4)$ gives
$$
\calG_T^{-1}(\Id) = \frac{1}{T}\Id - S + \frac{T}{12}R_A + \mathrm{O}(T^3).
$$
Here $e^{tA}Xe^{tA^\top} = e^{t\mathcal A}(X)$, so that $\calG_T = \sum_{q \geq 0}\frac{T^{q + 1}}{(q + 1)!}\mathcal A^q$ and $T\calG_T^{-1} = \varphi(T\mathcal A)$, where $\varphi(z) = z/(e^z - 1)$ is analytic on the disc $\vert z\vert<2\pi$. The displayed coefficients follow from $\mathcal A(\Id) = 2S$ and $\mathcal A^2(\Id) = R_A$, the term of order $T^2$ being absent because $\varphi$ has no cubic term.
The operator $\calG_T/T$ and its inverse are analytic near $T = 0$. The adjoint map has the same first-order expansion at $\Id$, so both inverse images appearing in Lemma~\ref{lem_inverse_certificate} are positive definite for small $T$. The lemma gives
$X_T = \frac{\calG_T^{-1}(\Id)}{\tr(\calG_T^{-1}(\Id))}$ and $\gamma_{\mathrm{rel}}(A,T) = \frac{1}{\tr(\calG_T^{-1}(\Id))}$.
Taking the trace in the refined inverse expansion gives $\tr(\calG_T^{-1}(\Id)) = n/T - \tr A + T\tr R_A/12 + \mathrm{O}(T^3)$. Expanding its reciprocal and normalizing the inverse image give \eqref{relaxed_jet} and \eqref{relaxed_optimizer}. Formula \eqref{relaxed_optimizer} makes $X_T$ positive definite for small $T$, which proves the last statement.
\end{proof}

\begin{remark}[what the higher-order terms detect]
Writing $K = (A - A^\top)/2$, one has
$R_A = 4S^2 + 2(KS - SK) = 4S^2 + AA^\top - A^\top A$ and therefore $\tr R_A = 4\tr(S^2)$. Thus the nonnormal commutator first enters the optimizer $X_T$ at order $T^2$, whereas it cancels from the optimal value through order $T^3$. These formulas also provide two checks. If $A^\top = -A$, then $R_A = 0$ and one recovers the exact identity $\gamma_{\mathrm{rel}}(A,T) = T/n$ in \eqref{skew_relaxed}, while the unique small-time relaxed optimizer is exactly $X_T = \Id/n$. If $A = A^\top$, then $R_A = 4A^2$ and the entries of $\calG_T^{-1}(\Id)$ in an eigenbasis of $A$ are $1/g_j(T) = 1/T - a_j + a_j^2T/3 + \mathrm{O}(T^3)$, in agreement with Proposition~\ref{prop_symmetric}.
\end{remark}

Since the relaxed maximizer is unique and has full rank, $\gamma_k(A,T)<\gamma_{\mathrm{rel}}(A,T)$ for every $k<n$ and every sufficiently small $T$. Thus the convexification has a strict gap at every lower rank. We now show that the scalar gap is asymptotically much larger than this rank statement alone suggests.

\subsection{Sharp small-time asymptotics at rank one}
We call a vector $b$ cyclic for $A$ if $\Span(b,Ab,\ldots,A^{n - 1}b) = \R^n$ (Kalman condition for the pair $(A,b)$), and call $A$ cyclic if it admits such a vector. For a unit cyclic vector, set
$$
d_n(A,b) = \dist(A^{n - 1}b,\Span(b,Ab,\ldots,A^{n - 2}b)),
$$
where the predecessor space is $\{0\}$ when $n=1$, and set $d_n(A,b) = 0$ when $b$ is not cyclic. Let
$$
\delta_n(A) = \max_{\Vert b\Vert = 1}d_n(A,b).
$$
For a cyclic pair, let $K_{A,b} = (b,Ab,\ldots,A^{n - 1}b)$ be its Kalman matrix, and write $K_{A,b} = QR$ for its QR factorization with positive diagonal in $R$. We denote by $q_n(A,b)$ the last column of $Q$.

The function $d_n(A,\cdot)$, extended by zero on the noncyclic set, is continuous on the unit sphere. Thus $\delta_n(A)$ is attained. If $A$ is cyclic, then $\delta_n(A)>0$, and the compact set
$$
\mathcal M(A) = \{b\in\R^n\ \mid\ \Vert b\Vert = 1,\ d_n(A,b) = \delta_n(A)\}
$$
consists of cyclic vectors. We may therefore define
\begin{equation}\label{alpha_n}
\alpha_n(A) = \max_{b\in\mathcal M(A)}q_n(A,b)^\top A q_n(A,b).
\end{equation}
This maximum is attained. It is not a second actuator-design problem: $\delta_n(A)$ is obtained by maximizing the leading coefficient of the original problem \eqref{gamma1_intro}, and $\alpha_n(A)$ resolves, at the next order, possible ties among these leading-order maximizers. No uniqueness is asserted.

\begin{theorem}[two-term optimal scalar design]\label{thm_scalar_jet}
Assume that $A$ is cyclic. As $T\to0^+$,
\begin{equation}\label{scalar_jet}
\gamma_1(A,T) = \kappa_n\, \delta_n(A)^2\, T^{2n - 1}\, (1 + \alpha_n(A)T + \mathrm{o}(T)),
\qquad
\kappa_n = \frac{((n - 1)!)^2}{(2n - 1)!(2n - 2)!}.
\end{equation}
Consequently,
\begin{equation}\label{ratio_jet}
\rho_1(A,T) = n\, \kappa_n\, \delta_n(A)^2\, T^{2n - 2} \Big(1 + (\alpha_n(A) - \frac{\tr A}{n})T + \mathrm{o}(T)\Big).
\end{equation}
\end{theorem}

\begin{proof}
For every fixed cyclic pair, the spectral expansion proved in \cite[Theorem~1.1]{TrelatZuazua2026gramian} gives
\begin{equation}\label{fixed_pair_jet}
\lambda_{\min}(G_T(A,b)) = \kappa_n\, d_n(A,b)^2\, T^{2n - 1} (1 + q_n(A,b)^\top Aq_n(A,b)T + \mathrm{O}(T^2)).
\end{equation}
The remainder is uniform on compact families on which $d_n$ stays bounded away from zero.

We first check that the optimizing vectors remain in such a family. Let cyclic vectors $b$ converge to a noncyclic vector $b_0$. A minimal polynomial relation for $b_0$, multiplied by a suitable power of $A$, gives a vector $z\in\ker K_{A,b_0}$ whose last coordinate is nonzero. If the last row $e_n^\top K_{A,b}^{-1}$ remained bounded, a subsequence would converge to a row $w^\top$ satisfying $w^\top K_{A,b_0} = e_n^\top$. Applying this identity to $z$ would give $z_n = 0$, a contradiction. Hence the last row becomes unbounded. Since $d_n(A,b) = 1/\Vert e_n^\top K_{A,b}^{-1}\Vert$, this also proves the asserted continuity of the extension of $d_n$. Moreover, the Rayleigh quotient in the last Krylov direction and the uniform expansion
$$
\langle q_n(A,b),e^{tA}b \rangle = d_n(A,b)\frac{t^{n - 1}}{(n - 1)!} + \mathrm{O}(t^n)
$$
give, uniformly on the unit sphere,
$\lambda_{\min}(G_T(A,b)) \leq C T^{2n - 1}(d_n(A,b)^2 + T^2)$.
Comparison with any fixed cyclic actuator shows that every maximizer $b_T$ in \eqref{gamma1_intro} satisfies $d_n(A,b_T) \geq \eta>0$ for all sufficiently small $T$.

We may therefore maximize the uniform expansion \eqref{fixed_pair_jet} on one fixed compact subset of cyclic vectors. Its leading term first forces every accumulation point of $(b_T)$ to belong to $\mathcal M(A)$. Comparing the value at $b_T$ with the value at a point of $\mathcal M(A)$ realizing \eqref{alpha_n}, and then dividing the maximality inequality by $T$, shows that every accumulation point also realizes $\alpha_n(A)$. The same comparison yields \eqref{scalar_jet}, with the remainder $\mathrm{o}(T)$ displayed there. Division by \eqref{relaxed_jet} gives \eqref{ratio_jet}.
\end{proof}

\begin{remark}[small-time selection and nonuniqueness]
Assume that $A$ is cyclic. For every $T>0$, let $b_T$ be any maximizer in the definition of $\gamma_1(A,T)$. Every accumulation point $b_0$ as $T\to0^+$ belongs to $\mathcal M(A)$ and satisfies $q_n(A,b_0)^\top Aq_n(A,b_0) = \alpha_n(A)$.
Thus $\alpha_n(A)$ describes the first selection exerted by the original finite-horizon problem; it neither defines a different design problem nor implies uniqueness or convergence of the whole family $(b_T)$. Literal uniqueness is impossible because $G_T(A,-b) = G_T(A,b)$. More generally, $G_T(A,Rb) = RG_T(A,b)R^\top$ for every orthogonal matrix $R$ commuting with $A$, and the theorem asserts no uniqueness modulo these symmetries. Distinct, inequivalent maximizing families may occur. If $A$ is not cyclic, then $\gamma_1(A,T) = 0$ and every unit vector is a maximizer for every $T>0$.
\end{remark}

\medskip
{\bf The case of normal matrices.}
For normal matrices, this selection can be made completely spectral. More is true: the first displacement of the exact maximizing weights makes their dependence on the horizon explicit.

\begin{proposition}\label{prop_normal_weights}
Assume that $A$ is real and normal and has simple complex spectrum. Choose a unitary eigenbasis $(v_1,\ldots,v_n)$ of $A_\C$, ordered so that $A_\C v_k = \lambda_kv_k$ and compatible with complex conjugation. Set
$$
\chi_A(z) = \det(z\,\Id - A), \qquad h_k = \vert\chi_A'(\lambda_k)\vert, \qquad H = \sum_{k = 1}^n\frac{1}{h_k}.
$$
For a unit real actuator $b$, define its spectral weights by $\pi_j(b) = \vert \langle b,v_j \rangle \vert^2$, $j\in\{1,\ldots,n\}$. They sum to one and coincide on conjugate eigenpairs.
The unique spectral-weight vector maximizing $d_n(A,b)$, hence the leading coefficient, is given by
\begin{equation}\label{normal_weights}
\pi_j^\star = \frac{1}{Hh_j} \qquad\forall j\in\{1,\ldots,n\}.
\end{equation}
For every sufficiently small $T>0$, all maximizers $b_T$ of the original problem \eqref{gamma1_intro} have the same spectral weights $\pi_j(T) = \vert \langle b_T,v_j \rangle \vert^2$. This exact finite-horizon weight vector is unique, depends analytically on $T$, and satisfies, as $T\to 0^+$,
\begin{equation}\label{normal_weight_displacement}
\pi_j(T) = \pi_j^\star\Big(1 + \frac{\alpha_n(A) - \operatorname{Re}\lambda_j}{2}T\Big) + \mathrm{O}(T^2).
\end{equation}
Thus $(\pi_j^\star)_{1 \leq j \leq n}$ belongs to the renormalized leading problem and should not be interpreted as a distinguished maximizer at $T = 0$, when every Gramian vanishes.
Moreover,
\begin{equation}\label{normal_invariants}
\delta_n(A) = \frac{1}{H},
\qquad
\alpha_n(A) = \frac{1}{H}\sum_{j = 1}^n\frac{\operatorname{Re}\lambda_j}{h_j}.
\end{equation}
In this normal case, the remainders $\mathrm{o}(T)$ in \eqref{scalar_jet} and \eqref{ratio_jet} improve to $\mathrm{O}(T^2)$.
\end{proposition}

\begin{proof}
In the chosen eigenbasis, diagonal phase changes commute with $A_\C$ and conjugate one Gramian into another. Hence its eigenvalues depend on $b$ only through $\pi_j = \vert \langle b,v_j \rangle \vert^2$. For a cyclic vector, Lagrange interpolation gives
$$
d_n(A,b)^2 = \bigg(\sum_{j = 1}^n\frac{1}{\pi_jh_j^2}\bigg)^{-1}, \qquad q_n(A,b)^\top Aq_n(A,b) = d_n(A,b)^2\sum_{j = 1}^n\frac{\operatorname{Re}\lambda_j}{\pi_jh_j^2}.
$$
The Cauchy-Schwarz inequality under $\sum_j\pi_j = 1$ shows that the leading term has the unique maximizing weight vector $\pi^\star$, and gives \eqref{normal_invariants}. For real actuators, the admissible simplex is the fixed set of the involution exchanging conjugate eigenvalues: $\pi_j = \pi_k$ whenever $\lambda_k = \overline{\lambda_j}$. Since $h_j = h_k$ on every conjugate pair, $\pi^\star$ belongs to this simplex. The exact normalized objective is invariant under the same involution, so at every fixed point its gradient has equal components on paired indices. Stationarity on the fixed simplex can therefore be written with one equation per index and one common multiplier, the equations on each conjugate pair being identical.

It remains to determine the displacement. Write $r_j = \operatorname{Re}\lambda_j$. For $\pi$ near the interior point $\pi^\star$ in the admissible simplex, choose a representative actuator $b(\pi)$ by analytic square roots in the real eigenlines and fixed phases in the rotation planes; every such nearby weight is positive, so $b(\pi)$ is cyclic and depends analytically on $\pi$. By \cite[Theorem~1.1]{TrelatZuazua2026gramian}, the function $T^{-(2n - 1)}\lambda_{\min}(G_T(A,b))$ extends jointly real analytically in $(T,A,b)$ near every controllable pair at $T = 0$. Composing with $\pi\mapsto b(\pi)$ shows that $(T,\pi)\mapsto T^{-(2n - 1)}\lambda_{\min}(G_T(A,b(\pi)))$ is real analytic near $(0,\pi^\star)$, so that the expansion \eqref{fixed_pair_jet} holds locally in the $C^2$ topology with respect to $\pi$.

At $\pi^\star$, the Hessian of $\log d_n(A,b)^2$ on the tangent space $\sum_j\xi_j = 0$ is
$$
D^2\log d_n^2(\pi^\star)[\xi,\xi] = -2H\sum_{j = 1}^nh_j\xi_j^2,
$$
and the gradient of the correction term has components $\alpha_n(A) - r_j$, modulo the constraint. The leading maximum is therefore nondegenerate. The analytic implicit function theorem applied to the normalized smallest eigenvalue gives a unique analytic critical branch near $\pi^\star$, which remains a strict local maximum for small $T$. Differentiating its stationarity equations at $T = 0$ yields
$-2Hh_j\pi_j'(0) + \alpha_n(A) - r_j = \mu$,
where $\mu$ is the derivative of the Lagrange multiplier. Summing these equations after division by $h_j$, and using $\sum_j\pi_j'(0) = 0$ and the second identity in \eqref{normal_invariants}, gives $\mu = 0$. Since $\pi_j^\star = 1/(Hh_j)$, this proves \eqref{normal_weight_displacement}.

Finally, the uniform leading expansion and the uniqueness of $\pi^\star$ force every globally maximizing weight vector to converge to $\pi^\star$ as $T\to0^+$. For small $T$, all global maximizers therefore lie in the neighborhood where the nondegenerate critical branch is unique. Joint analyticity also gives the stated $\mathrm{O}(T^2)$ improvements.
\end{proof}

When the spectrum is real, formula \eqref{normal_weights} is exactly the unique $c$-optimal weight allocation for estimating the leading coefficient in polynomial regression on the prescribed nodes $\lambda_1,\ldots,\lambda_n$; this is the finite-support instance of Elfving's theory \cite{Elfving1952}. Its role here is different: it is the unique leading spectral weight of the control problem, from which the exact finite-horizon branch and its displacement are obtained.

\begin{remark}[on the uniqueness of the normal maximizers]
Let $\mathcal Z_O(A) = \{R\in O(n)\mid RA = AR\}$. Fix a sufficiently small $T>0$ and let $b_T^0$ be one unit-vector maximizer of the scalar problem \eqref{gamma1_intro}. Then the full set of unit-vector maximizers of that problem is
$$
\{b\in\R^n\ \mid\ \Vert b\Vert = 1,\ \lambda_{\min}(G_T(A,b)) = \gamma_1(A,T)\} = \{Rb_T^0\,\mid\, R\in\mathcal Z_O(A)\}.
$$
Indeed, the commutant preserves the objective, and two real vectors with the same spectral weights differ by independent signs on the real eigenlines and rotations in the conjugate-pair planes. Thus the weight vector is unique, whereas the actuator is unique only modulo the orthogonal commutant. If the spectrum contains $r$ real eigenvalues and $q$ nonreal conjugate pairs, then $\mathcal Z_O(A)$ is isomorphic to $\{-1,1\}^r\times(S^1)^q$. In particular, a symmetric matrix with simple spectrum has $2^n$ maximizing vectors, or $2^{n - 1}$ unoriented actuator lines. An arbitrary choice of $b_T$ need not converge as $T\to0^+$ even though its unique weight vector $\pi(T)$ is analytic and convergent.
\end{remark}

\medskip
{\bf The small-time relaxation gap.}
Theorems~\ref{thm_relaxed_jet} and~\ref{thm_scalar_jet} reveal a sharp mismatch between the full-rank relaxed design and the best rank-one design. For every matrix $A$, the relaxed maximizer has rank $n$ at sufficiently small horizons. If $A$ is cyclic, one scalar actuator achieves only the fraction $\rho_1(A,T)\sim n\,\kappa_n\,\delta_n(A)^2\,T^{2n - 2}$ of the relaxed value. Equivalently, the best worst-case $L^2$ control norm obtainable with one actuator exceeds the relaxed benchmark by a factor equivalent to $\frac{1}{\sqrt{n\kappa_n}\delta_n(A)}T^{-(n - 1)}$.
At the level of Gramian performance, the relaxed-to-scalar ratio grows like $T^{-2n + 2}$. If $A^\top = -A$, then $\alpha_n(A) = \tr A = 0$, so the relative term of order $T$ in \eqref{ratio_jet} vanishes.

\medskip
{\bf Addition of a scalar drift.}
Replacing $A$ by $A + a\Id$, with $a\in\R$, leaves every Krylov space and every height $d_n(A,b)$ unchanged, and replaces $q_n(A,b)^\top Aq_n(A,b)$ by $q_n(A,b)^\top Aq_n(A,b) + a$, as observed in \cite{TrelatZuazua2026gramian} at the level of the whole spectrum. Hence $\delta_n(A + a\Id) = \delta_n(A)$ and $\alpha_n(A + a\Id) = \alpha_n(A) + a$. Since $\tr(A + a\Id)/n = \tr A/n + a$, both the leading coefficient and the first relative correction in \eqref{ratio_jet} are invariant. Thus, up to the order resolved here, the relaxation ratio detects only the anisotropic part of the drift, although the scalar and relaxed values separately depend on $a$. For normal matrices, formulas \eqref{normal_weights} and \eqref{normal_weight_displacement} show in addition that the leading optimal weights and their first displacement are unchanged.

\subsection{Rank-dependent small-time hierarchy}
The dichotomy between rank one and full rank is in fact a hierarchy. At every rank, the small-time scale is determined by the Kalman step of the input covariance. For a covariance $X$, set $\mathcal K_j(A,X) = \sum_{q = 0}^{j - 1}A^q\operatorname{Ran}X$ for $j \geq 1$, with $\mathcal K_0(A,X) = \{0\}$, and let
$$
s_k(A) = \min\{j \geq 1\ \mid\ \exists X\succeq0,\ \tr X = 1,\ \rank X \leq k,\ \mathcal K_j(A,X) = \R^n\}
$$
be the least Kalman step attainable with at most $k$ input directions, with $s_k(A) = +\infty$ when no admissible covariance is controllable.

\begin{proposition}[small-time gap at every rank]\label{prop_rank_gap}
Let $k\in\{1,\ldots,n\}$. If $s_k(A) = +\infty$, then $\gamma_k(A,T) = 0$ for every $T>0$. Otherwise, writing $s = s_k(A)$, there exist $T_0>0$ and $0<c_\star \leq C_\star$ such that
\begin{equation}\label{rank_gap}
c_\star \, T^{2s - 1} \leq \gamma_k(A,T) \leq C_\star \, T^{2s - 1}\qquad\forall T\in(0,T_0].
\end{equation}
Moreover, there exists a constant $c_k(A)>0$ such that
\begin{equation}\label{rank_gap_equivalent}
\gamma_k(A,T)\sim c_k(A)\,T^{2s - 1}
\qquad\hbox{as }T\to0^+.
\end{equation}
Consequently, $\frac{\gamma_k(A,T)}{\gamma_{\mathrm{rel}}(A,T)} \sim n\, c_k(A)\,T^{2s - 2}$.
If $A$ is cyclic, then $s_k(A) = \lceil n/k\rceil$.
\end{proposition}

\begin{proof}
If $s_k(A) = +\infty$, every admissible pair is uncontrollable, hence every corresponding Gramian is singular and $\gamma_k(A,T) = 0$. Assume from now on that $s = s_k(A)<+\infty$.

Let $X = BB^\top$ be admissible in \eqref{gammam}, normalized by $\tr(B^\top B) = 1$, so that $\Vert B\Vert \leq 1$. Since $\operatorname{Ran}B = \operatorname{Ran}X$, the spaces $\mathcal K_j(A,X)$ are precisely the multi-input Kalman spaces generated by $B$; in particular, the Kalman step depends on $X$ and not on the chosen factorization. By minimality of $s$, one has $\mathcal K_{s - 1}(A,X)\neq\R^n$; let $v$ be a unit vector orthogonal to it. Then $v^\top A^qB = 0$ for every $q \leq s - 2$, hence $\Vert B^\top e^{tA^\top}v\Vert \leq \frac{t^{s - 1}(1 + \Vert A\Vert)^{s - 1}}{(s - 1)!}e^{t\Vert A\Vert}$ and
$$
\lambda_{\min}(\calG_T(X)) \leq \int_0^T\Vert B^\top e^{tA^\top}v\Vert^2dt \leq \frac{(1 + \Vert A\Vert)^{2s - 2}e^{2T\Vert A\Vert}}{(2s - 1)((s - 1)!)^2}\,T^{2s - 1}.
$$
The right-hand side does not depend on $X$, which gives the upper bound in \eqref{rank_gap}.

For the lower bound, fix an admissible covariance $X_0 = B_0B_0^\top$ such that $\mathcal K_s(A,\allowbreak X_0) = \R^n$. By minimality of $s$, the pair $(A,B_0)$ has Kalman step exactly $s$; equivalently, $s - 1$ is the least integer $K$ with $\rank[B_0,AB_0,\ldots,A^KB_0] = n$. The classical fixed-pair estimate of \cite[(4.1), (4.2), (4.3)]{Seidman1988} then gives $\lambda_{\min}(\calG_T(X_0))^{-1/2} \sim \gamma_0\,T^{-(s - 1/2)}$ with $\gamma_0 > 0$, that is, $\lambda_{\min}(\calG_T(X_0)) \sim \gamma_0^{-2}\,T^{2s - 1}$. This yields the lower bound in \eqref{rank_gap}.

Finally, $\dim\mathcal K_j(A,X) \leq jk$ when $\rank X \leq k$, and hence $s_k(A) \geq \lceil n/k\rceil$. Conversely, assume that $A$ is cyclic, choose a cyclic vector $b$, and set $s = \lceil n/k\rceil$ and $m = \lceil n/s\rceil \leq k$. The input matrix $B = (b,A^sb,\ldots,A^{(m - 1)s}b)$ has rank at most $k$, and the powers generated by $B,AB,\ldots,A^{s - 1}B$ contain $b,Ab,\ldots,A^{n - 1}b$. After normalization, $X = BB^\top/\tr(B^\top B)$ is admissible and satisfies $\mathcal K_s(A,X) = \R^n$. Thus $s_k(A) = \lceil n/k\rceil$.

It remains to prove the existence of the limit in \eqref{rank_gap_equivalent}. The admissible set in \eqref{gammam} is compact and semialgebraic, the rank constraint being the vanishing of all minors of size $k + 1$. After an affine rescaling of a compact time interval and of a box containing the admissible set, the entries of $(T,X)\mapsto\calG_T(X)$ are restricted analytic functions, while the smallest eigenvalue is semialgebraic as a function of the entries of a symmetric matrix. It follows that $T\mapsto\gamma_k(A,T)$ is definable, with parameters depending on the fixed matrix $A$, in the o-minimal structure of restricted analytic functions; here we use stability of definability under maximization over a definably compact set. Therefore $T\mapsto \frac{\gamma_k(A,T)}{T^{2s - 1}}$
is definable and, by \eqref{rank_gap}, remains in a compact subinterval of $(0,+\infty)$. The o-minimal monotonicity theorem \cite{vandenDries1998} implies that it has a finite positive limit as $T\to0^+$, which is $c_k(A)$.
\end{proof}

For a cyclic matrix, $s_k(A) - 1 = \lceil n/k\rceil - 1$ is the least achievable maximal Kalman depth with at most $k$ input directions. The power-law exponent in the ratio therefore decreases from $2n - 2$ at rank one to $0$ at full rank; it equals $2$ whenever $\lceil n/2\rceil \leq k<n$. The coefficient $c_k(A)$ also admits a lower bound without any cyclicity assumption. Fix $k$ such that $s = s_k(A) < +\infty$. For every admissible covariance $X$ of rank at most $k$ and Kalman step $s$, the fixed-pair estimate used in the preceding proof shows that the limit $c(A,X) = \lim_{T\to0^+}T^{1 - 2s}\lambda_{\min}(\calG_T(X))$ exists and is positive. This constant is independent of the factorization of $X$, since the Gramian itself depends only on $X$, and Seidman's graded construction gives an explicit formula for it \cite[(3.4), (4.1), (4.2), (4.3)]{Seidman1988}. Since $\gamma_k(A,T) \geq \lambda_{\min}(\calG_T(X))$ for every such fixed covariance, taking limits and then the supremum gives $c_k(A) \geq \sup_X c(A,X)$, where $X$ ranges over admissible covariances of rank at most $k$ and Kalman step $s_k(A)$. At $k = n$, for every $A$ one has $s_n(A) = 1$ and $c(A,X) = \lambda_{\min}(X)$, so the supremum is $1/n = c_n(A)$. When $A$ is cyclic, the lower bound is also exact at $k = 1$: \eqref{fixed_pair_jet} gives $c(A,bb^\top) = \kappa_n\,d_n(A,b)^2$, hence its supremum is $\kappa_n\,\delta_n(A)^2 = c_1(A)$ by Theorem~\ref{thm_scalar_jet}. Whether equality persists at intermediate ranks is open; the issue is whether $T$-dependent designs approaching a locus where the Kalman growth profile degenerates can exploit nonuniformity in the fixed-pair asymptotics. The fixed-pair scale $T^{2s - 1}$ is classical \cite{Seidman1988}, as is the role of controllability indices \cite{Brunovsky1970}. The design-level estimate \eqref{rank_gap} follows here from an upper bound uniform over all covariances of rank at most $k$, combined with one fixed covariance attaining the least possible Kalman step.

\subsection{Symmetric matrices: a permanent obstruction}
The preceding obstruction to exact low-rank recovery is universal but confined to small time. We now show that, for symmetric matrices, it persists at every horizon $T>0$.

\begin{proposition}[unique relaxed maximizer for symmetric matrices]\label{prop_symmetric}
Let $A = A^\top$ and $n \geq 2$. Choose any orthonormal eigenbasis $(e_1,\ldots,e_n)$ such that $Ae_j = a_je_j$, and set $g_j(T) = \int_0^T e^{2a_jt}\,dt$.
Then
\begin{equation}\label{symmetric_relaxed}
\gamma_{\mathrm{rel}}(A,T) = \bigg(\sum_{j = 1}^n\frac{1}{g_j(T)}\bigg)^{-1}.
\end{equation}
In this eigenbasis, the unique relaxed maximizer is
\begin{equation}\label{symmetric_covariance}
X_\star = \diag\left(\frac{\gamma_{\mathrm{rel}}(A,T)}{g_1(T)},\ldots,\frac{\gamma_{\mathrm{rel}}(A,T)}{g_n(T)}\right).
\end{equation}
Consequently, $r(A,T) = n$ and $\gamma_k(A,T)<\gamma_{\mathrm{rel}}(A,T)$ for every $k<n$ and every $T>0$.
\end{proposition}

\begin{proof}
For every feasible covariance $X$, written in the chosen eigenbasis,
$$
\lambda_{\min}(\calG_T(X)) \leq \min_jg_j(T)X_{jj} \leq \bigg(\sum_{j = 1}^n\frac{1}{g_j(T)}\bigg)^{-1}.
$$
The covariance \eqref{symmetric_covariance} has trace one and satisfies $\calG_T(X_\star) = \gamma_{\mathrm{rel}}(A,T)\,\Id$, proving \eqref{symmetric_relaxed}.

If another covariance $X$ were a maximizer, then $\calG_T(X)\succeq\gamma_{\mathrm{rel}}(A,T)\,\Id$. Its diagonal entries would give $X_{jj} \geq \gamma_{\mathrm{rel}}(A,T)/g_j(T)$. Their lower bounds already sum to one, so equality holds for every $j$. The positive semidefinite matrix $\calG_T(X) - \gamma_{\mathrm{rel}}(A,T)\,\Id$ has zero diagonal and is therefore zero. Since its off-diagonal entries are $X_{jk}\int_0^T e^{(a_j + a_k)t}\,dt$ and the integral is positive, $X_{jk} = 0$ for $j\neq k$. Thus $X = X_\star$, which is positive definite. Formula \eqref{rank_equivalence} gives the last assertion. If eigenvalues are repeated, the corresponding diagonal coefficients in \eqref{symmetric_covariance} coincide, so the formula is independent of the eigenbasis chosen inside each eigenspace.
\end{proof}

We have found two distinct obstructions to exact low-rank recovery. For an arbitrary matrix, full rank is forced at sufficiently small horizons; for a symmetric matrix, it is forced at every horizon. We now turn to skew-symmetric matrices, for which time averaging can instead close the gap.

\section{Skew-symmetric dynamics: exact recovery}\label{sec_skew}
We now ask when a low-rank input covariance $X$, and in particular a rank-one covariance $X = bb^\top$ generated by a single actuator, attains the relaxed maximum.

Assume throughout this section that $A^\top = -A$. The matrix $e^{tA}$ is orthogonal, and thus $\tr(e^{tA}Xe^{tA^\top}) = \tr(e^{tA}e^{tA^\top}X) = \tr X$.
Hence every covariance of trace one satisfies $\tr\calG_T(X) = T$. It follows that $\lambda_{\min}(\calG_T(X)) \leq T/n$, whereas $X = \Id/n$ gives equality. Therefore
\begin{equation}\label{skew_relaxed}
\gamma_{\mathrm{rel}}(A,T) = \frac{T}{n}.
\end{equation}
Since a positive semidefinite matrix whose smallest eigenvalue equals its average eigenvalue is scalar,
\begin{equation}\label{isotropic_equivalence}
\lambda_{\min}(\calG_T(X)) = \frac{T}{n} \ \Leftrightarrow\ \calG_T(X) = \frac{T}{n}\Id.
\end{equation}
Thus exact recovery of the relaxed maximum, equivalently the absence of a gap at the prescribed rank, means that time averaging makes the input covariance isotropic. In control terms, the minimum control energy is then the same in every unit target direction.

The planar rotation makes this isotropic averaging explicit. Let
$$
J = \begin{pmatrix}0&-1\\1&0\end{pmatrix}.
$$

\begin{example}\label{ex_planar_rotation}
Let $A = \omega J$, with $\omega\neq0$. Then
$\lambda_{\min}(G_T(\omega J,b)) = \frac{T}{2} - \frac{\vert\sin(\omega T)\vert}{2\vert\omega\vert}$, for every unit vector $b$.
Every direction is a maximizer. The relaxed value is $T/2$, and scalar recovery is exact precisely at the half-periods $T\in\frac{\pi}{\vert\omega\vert}\N^*$. At these horizons the rotating line generated by $b$ averages equally in both planar directions.
\end{example}

\subsection{Recovery with one actuator}

Recall that a matrix is cyclic when it admits a cyclic vector. For a real skew-symmetric matrix this is equivalent to simplicity of its complex spectrum. In that case there is an orthogonal decomposition
\begin{equation}\label{real_blocks}
\R^n = V_0\oplus\bigoplus_{\ell = 1}^L V_\ell,
\end{equation}
where $V_0 = \ker A$ has dimension zero or one, every $V_\ell$ is a two-dimensional invariant plane, and $A_{\vert V_\ell} = \omega_\ell J$ for distinct positive frequencies $\omega_1,\ldots,\omega_L$. We denote the corresponding orthogonal projectors by $P_0,P_1,\ldots,P_L$.

\begin{theorem}[exact scalar recovery]\label{thm_scalar_recovery}
Let $A^\top = -A$. Then $\gamma_1(A,T) = \gamma_{\mathrm{rel}}(A,T)$ if and only if $A$ is cyclic and $e^{TA} = \sigma\,\Id$ for some $\sigma\in\{-1,1\}$. If $\ker A\neq\{0\}$, necessarily $\sigma = 1$.

When these conditions hold, the unit actuators maximizing the scalar problem \eqref{gamma1_intro} are exactly those whose modal weights, namely the fractions of the squared actuator norm carried by the invariant frequency subspaces in \eqref{real_blocks}, satisfy
\begin{equation}\label{optimal_modal_weights}
\Vert P_0b\Vert^2 = \frac{1}{n}\quad\hbox{if }\dim V_0 = 1,
\qquad
\Vert P_\ell b\Vert^2 = \frac{2}{n}\quad\forall \ell\in\{1,\ldots,L\}.
\end{equation}
The maximizing modal weights are unique, but the actuator is not: its direction inside each rotation plane and its sign on $V_0$ are arbitrary. Every such actuator satisfies $G_T(A,b) = T\,\Id/n$.
\end{theorem}

This includes $n = 1$, where $A = 0$ and $G_T(A,b) = T$ for every unit actuator.

\begin{proof}
Assume first that a unit vector $b$ attains the relaxed value. By \eqref{isotropic_equivalence}, $G_T(A,b) = T\,\Id/n$. Set $U = e^{TA}$. The finite-time Lyapunov identity is
\begin{equation}\label{lyapunov_identity}
AG_T + G_TA^\top = Ubb^\top U^\top - bb^\top.
\end{equation}
Its left-hand side vanishes, so $Ubb^\top U^\top = bb^\top$ and $Ub = \sigma b$ for some $\sigma\in\{-1,1\}$. The Gramian is positive definite, hence $b$ is cyclic for $A$. Since $U$ commutes with $A$, $UA^kb = A^kUb = \sigma A^kb$ for every $k \geq 0$. The Krylov vectors span $\R^n$, and therefore $U = \sigma\,\Id$. On $\ker A$, the endpoint map is the identity, which excludes $\sigma = -1$ when the kernel is nontrivial.

Conversely, suppose that $A$ is cyclic and $e^{TA} = \sigma\,\Id$. In the complexification choose unit eigenvectors $v_\omega$ satisfying $A_\C v_\omega = i\omega v_\omega$ and $v_{-\omega} = \overline{v_\omega}$, together with a real unit vector of $V_0$ when the kernel is present. All signed frequencies are distinct. The endpoint condition implies $\int_0^T e^{i(\omega - \nu)t}\,dt = 0$ for distinct frequencies $\omega$ and $\nu$. Thus all cross-correlations vanish. Conditions \eqref{optimal_modal_weights} put squared coefficient $1/n$ on every complex eigendirection, so all diagonal entries of the Gramian equal $T/n$. They are also necessary, which proves the characterization.
\end{proof}

\begin{remark}[only the modal weights matter]
For cyclic skew-symmetric $A$, the spectrum of $G_T(A,b)$ depends on $b$ only through $\Vert P_0b\Vert^2,\Vert P_1b\Vert^2,\ldots,\Vert P_Lb\Vert^2$. Indeed, rotations in the planes $V_\ell$, and a sign on $V_0$, commute with $A$ and conjugate the corresponding Gramians orthogonally. Hence the maximization defining $\gamma_1(A,T)$ reduces to a simplex of modal weights; directions inside the planes do not affect its value. At a scalar recovery horizon, all maximizers form one orbit under the orthogonal commutant of $A$. Away from resonance, neither uniqueness of the maximizing weight vector nor a closed formula is asserted.
\end{remark}

\subsection{The recovery rank}
If one actuator cannot attain $T/n$, how many are needed? This is precisely the recovery rank defined in \eqref{recovery_rank}. Partial resonances lead to intermediate answers. For example, let $A = J\oplus2J$ on $\R^4$. At a generic horizon the four signed frequencies have distinct endpoint phases and $r(A,T) = 4$. At $T = \pi$, the phases form two classes and $r(A,T) = 2$. At $T = 2\pi$, all phases coincide and $r(A,T) = 1$.

To state the general formula, let $A_\C$ be the complexification of an arbitrary real skew-symmetric matrix and write
\begin{equation}\label{spectral_decomposition}
\C^n = \bigoplus_{\omega\in\Omega}E_\omega,
\qquad
E_\omega = \ker(A_\C - i\omega\Id),
\qquad
\mu_\omega = \dim_\C E_\omega.
\end{equation}
At time $T$, group the frequencies by their endpoint phases:
$$
\Omega_\zeta(T) = \{\omega\in\Omega\,\mid\, e^{i\omega T} = \zeta\}, \qquad \zeta\in\sigma(e^{TA_\C}),
$$
and set
\begin{equation}\label{phase_multiplicity}
m_\zeta(T) = \max_{\omega\in\Omega_\zeta(T)}\mu_\omega.
\end{equation}
Different phase classes require orthogonal input subspaces. Inside one class, the smallest possible input dimension is the largest spectral multiplicity occurring there. These two facts give the following exact formula.

\begin{theorem}[recovery rank at every horizon]\label{thm_phase_rank}
Let $A^\top = -A$ and $T>0$. Then
\begin{equation}\label{rank_formula}
r(A,T) = \sum_{\zeta\in\sigma(e^{TA_\C})}m_\zeta(T).
\end{equation}
Equivalently, for every $k\in\{1,\ldots,n\}$,
$$
\gamma_k(A,T) = \gamma_{\mathrm{rel}}(A,T) = \frac{T}{n} \ \Leftrightarrow\ k \geq \sum_{\zeta\in\sigma(e^{TA_\C})}m_\zeta(T).
$$
In particular, $r(A,T) = n$ if all endpoint phases are distinct. If all phases coincide, then $r(A,T) = \max_\omega\mu_\omega$.
\end{theorem}

\begin{proof}
Let $P_\omega$ be the orthogonal projector onto $E_\omega$. In a unitary eigenbasis of $A_\C$,
one has
$P_\omega\calG_T(X)P_\nu = \Phi_T(\omega - \nu)P_\omega XP_\nu$, where $\Phi_T(s) = \int_0^T e^{ist}dt$.
Now $\Phi_T(0) = T$, while for $s\neq0$ it vanishes exactly when $e^{isT} = 1$. Hence $\calG_T(X) = T\,\Id/n$ is equivalent to
\begin{equation}\label{block_constraints}
P_\omega XP_\omega = \frac{1}{n}P_\omega,
\qquad
P_\omega XP_\nu = 0
\quad\hbox{if }e^{i\omega T}\neq e^{i\nu T}.
\end{equation}
Thus $X$ is an orthogonal sum over endpoint phase classes. Factor one phase block of $nX$ as a Gram matrix. Its diagonal restrictions assign an orthonormal family of size $\mu_\omega$ to every $E_\omega$, so its rank is at least $m_\zeta(T)$. Ranks of distinct phase blocks add. Conversely, Lemma~\ref{lem_real_realization} constructs a real covariance satisfying \eqref{block_constraints} with rank equal to the sum in \eqref{rank_formula}. The result follows from \eqref{isotropic_equivalence} and \eqref{rank_equivalence}.
\end{proof}

\begin{remark}[complete-period orbit frames]
If $e^{TA} = \pm\Id$, all endpoint phases coincide and \eqref{rank_formula} reduces to the largest spectral multiplicity. This is the circle-representation instance of the compact-group multiplicity rule in \cite[Theorem~6.1]{Iverson2018}. Formula \eqref{rank_formula} adds the sum over the phase classes which remain distinct on an arbitrary finite window.
\end{remark}

\begin{remark}[rank does not imply uniqueness]
Theorem~\ref{thm_phase_rank} determines the least recovery rank, not a unique recovering covariance. If all endpoint phases are distinct, \eqref{block_constraints} forces $X = \Id/n$, which is the unique relaxed maximizer. If a phase class contains several frequencies, the cross-blocks inside that class are not seen by the averaging operator and recovering covariances are generally nonunique.
\end{remark}

\section{Near recovery and long-time averaging}\label{sec_near}

Theorem~\ref{thm_scalar_recovery} shows that exact scalar recovery is resonant. Approximate recovery raises two separate questions. Is exact recovery stable at a fixed horizon? Can the relaxation ratio approach one without approaching a prescribed resonance? The next two results answer both questions.

\begin{theorem}[rigidity of near recovery]\label{thm_rigidity}
Let $A^\top = -A$, let $b$ be a unit vector, and assume that
\begin{equation}\label{near_optimal}
\lambda_{\min}(G_T(A,b)) \geq (1 - \varepsilon)\frac{T}{n},
\end{equation}
for some $\varepsilon\in[0,1)$.
Let $P_0$ be the orthogonal projector onto $\ker A$. Then $(A,b)$ is controllable, and there exists $\sigma\in\{-1,1\}$ such that
\begin{equation}\label{near_return}
\Vert e^{TA} - (P_0 + \sigma(\Id - P_0))\Vert
\leq \sqrt{\frac{n}{1 - \varepsilon}}\Vert A\Vert T\varepsilon.
\end{equation}
If $\dim\ker A = 1$ and the right-hand side of \eqref{near_return} is smaller than $\sqrt2$, then necessarily $\sigma = 1$.
\end{theorem}

\begin{proof}
Set $G = G_T(A,b)$ and $U = e^{TA}$. Since $\tr G = T$, assumption \eqref{near_optimal} implies
$\lambda_{\max}(G) - \lambda_{\min}(G) \leq T\varepsilon$.
In particular, $G$ is positive definite, so $(A,b)$ is controllable and $A$ is cyclic. The decomposition \eqref{real_blocks} is therefore available. If $\pi_\ell = \Vert P_\ell b\Vert^2$, then $\tr(P_\ell G) = T\pi_\ell$. The lower bound on $G$ gives $\pi_\ell \geq \frac{2(1 - \varepsilon)}{n}$ on every rotation plane.
If $\dim\ker A = 1$, the same argument gives $\pi_0 = \Vert P_0b\Vert^2 \geq (1 - \varepsilon)/n$.

The Lyapunov identity \eqref{lyapunov_identity} and skew-symmetry give $Ubb^\top U^\top - bb^\top = [A,G]$.
After subtracting from $G$ the midpoint of its extreme eigenvalues,
\begin{equation}\label{projector_defect}
\Vert Ubb^\top U^\top - bb^\top\Vert \leq \Vert A\Vert T\varepsilon.
\end{equation}
For unit vectors $v,w$, there is a sign $\sigma$ such that $\Vert v - \sigma w\Vert \leq \sqrt2\Vert vv^\top - ww^\top\Vert$.
Apply this to $v = Ub$ and $w = b$. On each $V_\ell$, the quantity $\Vert(U - \sigma\,\Id)v\Vert$ is independent of the unit vector $v\in V_\ell$. Combining this observation with $\pi_\ell \geq \frac{2(1 - \varepsilon)}{n}$ and \eqref{projector_defect} gives
$\Vert Ub - \sigma b\Vert^2 \geq \pi_\ell\Vert U_{\vert V_\ell} - \sigma\,\Id_{V_\ell}\Vert^2$
and therefore
$\Vert U_{\vert V_\ell} - \sigma\,\Id_{V_\ell}\Vert \leq \sqrt{\frac{n}{1 - \varepsilon}}\Vert A\Vert T\varepsilon$
for every $\ell$. Since $U = \Id$ on $\ker A$, taking the maximum over the orthogonal blocks proves \eqref{near_return}.

If the kernel is one-dimensional and $\sigma = -1$, then
$\Vert Ub + b\Vert^2 \geq 4\pi_0 \geq \frac{4(1 - \varepsilon)}{n}$,
whereas \eqref{projector_defect} gives $\Vert Ub + b\Vert^2 \leq 2\Vert A\Vert^2T^2\varepsilon^2$. These inequalities are incompatible when the right-hand side of \eqref{near_return} is smaller than $\sqrt2$.
\end{proof}

If $A$ is invertible, \eqref{near_return} says directly that the endpoint flow is close to $\Id$ or $-\Id$. If the kernel is one-dimensional, it is close to the identity as soon as the displayed smallness condition holds. The point of the estimate is its scale: a defect $\varepsilon = \mathrm{o}(1/(\Vert A\Vert T))$ forces an actual near return of the flow.

Read in the contrapositive, Theorem~\ref{thm_rigidity} produces a lower bound on the scalar gap at every horizon, resonant or not. Let
$$
D(A,T) = \min_{\sigma\in\{-1,1\}}\Vert e^{TA} - (P_0 + \sigma(\Id - P_0))\Vert
$$
be the distance from the endpoint flow to the two comparison maps in Theorem~\ref{thm_rigidity}. It vanishes at every scalar recovery horizon. The converse holds when $A$ is cyclic and invertible. Without cyclicity, a return does not make any scalar pair controllable; if $\ker A\neq\{0\}$, the branch $\sigma = -1$ may also vanish although scalar recovery is impossible.

\begin{corollary}[quantitative gap away from resonance]\label{cor_gap}
Let $A^\top = -A$ be nonzero and let $T>0$. Then
\begin{equation}\label{gap_lower_bound}
1 - \rho_1(A,T) \geq
\frac{2D(A,T)}{D(A,T) + \sqrt{D(A,T)^2 + 4n\Vert A\Vert^2T^2}}.
\end{equation}
\end{corollary}

\begin{proof}
Set $\varepsilon = 1 - \rho_1(A,T)\in[0,1]$ and let $b$ maximize \eqref{gamma1_intro}. Then $\lambda_{\min}(G_T(A,b)) = \gamma_1(A,T)$, which equals $(1 - \varepsilon)T/n$ by \eqref{skew_relaxed}. If $\varepsilon = 1$, inequality \eqref{gap_lower_bound} is immediate. If $\varepsilon<1$, Theorem~\ref{thm_rigidity} gives $D(A,T)^2(1 - \varepsilon) \leq n\Vert A\Vert^2T^2\varepsilon^2$.
Solving this quadratic inequality for $\varepsilon$ gives \eqref{gap_lower_bound}.
\end{proof}

Thus a single actuator is quantitatively suboptimal on every window on which the flow has not almost returned, the loss being at least of the order of the return defect divided by $\Vert A\Vert T$. The estimate becomes vacuous when $\Vert A\Vert T\to+\infty$, which is the regime in which averaging takes over, as the next proposition shows. The planar calculation in the next remark makes the constants in this comparison explicit.

\begin{remark}[optimality of the scale]
As in Example~\ref{ex_planar_rotation}, take $A = \omega J$ on $\R^2$, so that $\Vert A\Vert = \vert\omega\vert$ and $P_0 = 0$. Writing $s = \vert\sin(\omega T/2)\vert$ and $c = \vert\cos(\omega T/2)\vert$, one has $D(A,T) = 2\min(s,c)$, whereas Example~\ref{ex_planar_rotation} gives $1 - \rho_1(A,T) = \vert\sin(\omega T)\vert/(\vert\omega\vert T) = 2sc/(\vert\omega\vert T)$. Since $\max(s,c)$ ranges in $[1/\sqrt2,1]$,
$$
\frac{D(A,T)}{\sqrt2\,\Vert A\Vert T} \leq 1 - \rho_1(A,T) \leq \frac{D(A,T)}{\Vert A\Vert T}.
$$
On this planar family, $D(A,T) \leq \Vert A\Vert T$, so the right-hand side of \eqref{gap_lower_bound} is at least $D(A,T)/(2\Vert A\Vert T)$. Thus, as $T$ approaches any positive recovery horizon, the ratio of the exact gap to the lower bound in \eqref{gap_lower_bound} tends to $\sqrt2$; whenever $D(A,T)>0$, this ratio is at most $2$, and this upper bound is sharp since the ratio tends to $2$ as $\vert\omega\vert T\to0$.
\end{remark}

Long-time averaging is different. Let $A$ be cyclic and list its distinct signed frequencies as $\omega_1,\ldots,\omega_n$ in a complex eigenbasis, including zero if the kernel is nontrivial. Define the finite constant
$$
C_A = \max_{1 \leq j \leq n}\sum_{k\neq j}\frac{2}{\vert\omega_j - \omega_k\vert}.
$$
The denominators do not vanish because cyclicity of a skew-symmetric matrix is equivalent to simplicity of its complex spectrum.

\begin{proposition}[long-time averaging]
If $A^\top = -A$ is cyclic, then $\rho_1(A,T) \geq \max\{0,1 - \frac{C_A}{T}\}$.
In particular, $1 - \rho_1(A,T) = \mathrm{O}(T^{-1})$ as $T\to+\infty$.
\end{proposition}

\begin{proof}
Choose a real unit actuator whose squared coefficient on every complex eigendirection is $1/n$. In the complex eigenbasis its Gramian has diagonal entries $T/n$. Its off-diagonal entries have modulus
$$
\frac{1}{n}\left\vert\int_0^T e^{i(\omega_j - \omega_k)t}dt\right\vert \leq \frac{2}{n\vert\omega_j - \omega_k\vert}.
$$
Gershgorin's theorem gives $\lambda_{\min}(G_T(A,b)) \geq (T - C_A)/n$. Combining this with $\rho_1(A,T) \geq 0$ and dividing by \eqref{skew_relaxed} proves the result.
\end{proof}

Thus a very small defect on the scale $1/(\Vert A\Vert T)$ forces a near resonance. The long-time estimate controls the defect only at the scale $\mathrm{O}(1/T)$, which is not sufficient to force a near return through Theorem~\ref{thm_rigidity}. The ratio may therefore converge to one at long times without the endpoint map converging to one fixed common sign.

\section{Modal consequences for Galerkin truncations}\label{sec_PDE}

We finally ask what the finite-dimensional results predict for controlled partial differential equations. A separated source $b(x)u(t)$, with a fixed spatial actuator $b$ and a scalar time control $u$, is usually called a lumped control; see, e.g., \cite{PrivatTrelatZuazua2017,GeshkovskiZuazua2022}. It is the infinite-dimensional counterpart of a single input direction: projection onto the first $N$ eigenmodes gives the vector $b_N = (b_1,\ldots,b_N)^\top$, driven by the same scalar control $u(t)$. Whenever a consistent modal limit exists, it is naturally interpreted as a possibly distributional spatial profile $b$ in a separated source $b\,u(t)$.

Let $\varphi_j(x) = \sqrt{\frac{2}{\pi}}\sin(jx)$, $j \geq 1$, be the Dirichlet eigenbasis on $(0,\pi)$, and write formally $b = \sum_{j \geq 1}b_j\varphi_j$. Here $b$ denotes the spatial actuator profile. We do not assume a priori that $b\in L^2(0,\pi)$: its regularity is determined by its modal coefficients, and some limiting profiles below exist only as distributions. We use the spectral notation
\begin{equation}\label{HDs}
b\in H_D^{-s}\ \Leftrightarrow\ \sum_{j \geq 1}j^{-2s}\vert b_j\vert^2<+\infty.
\end{equation}

The three formal controlled equations that we consider in this section are $y_t - y_{xx} = b(x)u(t)$ (lumped-controlled heat equation), $iy_t + y_{xx} = ib(x)u(t)$ (lumped-controlled Schr\"odinger equation) and $y_{tt} - y_{xx} = b(x)u(t)$ (lumped-controlled wave equation), with homogeneous Dirichlet boundary conditions. The control $u$ is real for the heat and wave equations and complex for the Schr\"odinger equation. The harmless factor $i$ in the Schr\"odinger input ensures that its modal input vector is exactly $b$ under the convention used below.

Our procedure is finite dimensional. We first project each equation onto its first $N$ modes and solve the resulting actuator-design problem at fixed $N$. We then let $T\to0^+$ or choose a distinguished finite horizon, and only afterwards examine the modal profile as $N\to+\infty$. The maximizing profiles may depend on $N$ and need not be the projections of one fixed admissible PDE actuator. The resulting infinite sequences are renormalized modal signatures, not normalized maximizers for the full partial differential equations. A full-PDE statement would require a prescribed admissible space for $b$ and estimates uniform in $N$; the small-time regime would additionally require a controlled joint limit in $N$ and $T$. For the heat and Schr\"odinger truncations below, $\Vert A_N\Vert = N^2$, which underscores that none of the fixed-$N$ small-time remainders is asserted to be uniform as $N\to+\infty$. Table~\ref{tab_modal_comparison} summarizes the profiles and regularity thresholds obtained below.

\begin{table}[ht]
\centering
\small
\begin{tabular}{@{}>{\raggedright\arraybackslash}p{.22\textwidth}>{\raggedright\arraybackslash}p{.21\textwidth}>{\raggedright\arraybackslash}p{.25\textwidth}>{\raggedright\arraybackslash}p{.12\textwidth}@{}}
\hline
Model and regime & Modal profile & Interpretation & Threshold\\ \hline
Heat and Schr\"odinger, small time & concentration near $j\sim\sqrt N$ & dipole-type spectral signature & $s>3/2$\\
Schr\"odinger, revival & flat moduli & phase-aligned point-source signature & $s>1/2$\\
Wave, common period, energy metric & flat moduli & velocity-forced point-source signature & $s>1/2$\\
Wave, common period, unweighted metric & moduli proportional to $j$ & metric-dependent dipole signature & $s>3/2$\\ \hline
\end{tabular}
\caption{Finite-dimensional modal signatures. The thresholds concern renormalized formal infinite-mode profiles, not optimizers for the full PDEs.}
\label{tab_modal_comparison}
\end{table}

\subsection{Heat: concentration at intermediate frequencies}
For the first $N$ modes of the Dirichlet heat equation,
\begin{equation}\label{heat_truncation}
A_N = -\diag(1^2,2^2,\ldots,N^2).
\end{equation}
Proposition~\ref{prop_symmetric} rules out exact low-rank recovery at every horizon. At small time, Proposition~\ref{prop_normal_weights} gives the leading scalar weights and their first displacement explicitly.

\begin{proposition}[heat weights and their modal limit]\leavevmode\par
Fix $N$, and let $b_N(T) = \sum_{j = 1}^Nb_{N,j}(T)\varphi_j$ be any unit profile maximizing the scalar problem for \eqref{heat_truncation}. For every sufficiently small $T>0$, the squared coefficients of all such profiles coincide and, as $T\to0^+$, satisfy
\begin{equation}\label{heat_weights}
\begin{aligned}
\vert b_{N,j}(T)\vert^2 &= w_{N,j}(1 + \frac{2j^2 - 3N + 1}{4}T) + \mathrm{O}(T^2),\\
w_{N,j} &= \frac{j^2\binom{2N}{N + j}}{N2^{2N - 2}} \qquad\forall j\in\{1,\ldots,N\}.
\end{aligned}
\end{equation}
The exact weight vector is unique, but the profile is not: its $N$ signs are arbitrary, giving $2^N$ maximizing vectors or $2^{N - 1}$ unoriented actuator lines. The numbers $w_{N,j}$ are positive and sum to one. As $N\to+\infty$, the probability law of $j/\sqrt N$ with weights $w_{N,j}$ converges weakly on $(0,+\infty)$ to the Maxwell density
\begin{equation}\label{heat_bulk}
\frac{4}{\sqrt\pi}x^2e^{-x^2}dx
\end{equation}
with scale parameter $1/\sqrt2$ (law of $\vert Z\vert/\sqrt2$ for a standard Gaussian vector $Z\in\R^3$).
If $b_N^0$ is any leading profile satisfying $\vert b_{N,j}^0\vert^2 = w_{N,j}$, then $b_N^0$ converges weakly to zero in $L^2(0,\pi)$, while, for every fixed $j\geq1$,
\begin{equation}\label{heat_ratio}
\frac{\vert b_{N,j}^0\vert}{\vert b_{N,1}^0\vert}\underset{N\to+\infty}{\longrightarrow} j.
\end{equation}
Thus division by the first modal coefficient reveals a linear spectral envelope. A formal limiting spatial actuator profile $b = \sum_{j \geq 1}b_j\varphi_j$ with coefficient moduli proportional to $j$ belongs to $H_D^{-s}$ if and only if $s>3/2$.

The exact leading invariant and the first correction are
\begin{equation}\label{heat_invariants}
\delta_N(A_N) = \frac{(2N)!}{N2^{2N - 1}},
\qquad
\alpha_N(A_N) = -\frac{3N - 1}{2}.
\end{equation}
Consequently, for every fixed $N$,
\begin{equation}\label{heat_ratio_jet}
\rho_1(A_N,T)
= \frac{N(2N - 1)((N - 1)!)^2}{16^{N - 1}}T^{2N - 2} \Big(1 + \frac{(N - 1)(N - 2)}{3}T + \mathrm{O}(T^2)\Big).
\end{equation}
\end{proposition}

\begin{proof}
The characteristic polynomial is $\chi_{A_N}(z) = \prod_{k = 1}^N(z + k^2)$, and
$$
\vert\chi_{A_N}'(-j^2)\vert = \prod_{k\neq j}\vert j^2 - k^2\vert = \frac{(N - j)!(N + j)!}{2j^2}.
$$
Formula \eqref{normal_weights} and the identity
$\sum_{j = 1}^N\frac{2j^2}{(N - j)!(N + j)!} = \frac{N2^{2N - 1}}{(2N)!}$
give the formula for $w_{N,j}$ in \eqref{heat_weights} and the first formula in \eqref{heat_invariants}.

Let $K_N$ be binomial with parameters $(2N,1/2)$ and set $Y_N = (K_N - N)/\sqrt N$. For every bounded continuous function $h$ on $[0,+\infty)$, one has
$\sum_{j = 1}^Nw_{N,j}h(j/\sqrt N) = 2\mathbb E(Y_N^2h(\vert Y_N\vert))$.
The central limit theorem and uniform integrability of $Y_N^2$ yield \eqref{heat_bulk}. Moreover,
$$
\frac{w_{N,j + 1}}{w_{N,j}} = \frac{(j + 1)^2}{j^2}\frac{N - j}{N + j + 1} \underset{N\to+\infty}{\longrightarrow}\frac{(j + 1)^2}{j^2},
$$
proving \eqref{heat_ratio}. For fixed $j$, the central binomial estimate $\binom{2N}{N + j} = \mathrm{O}(4^N/\sqrt N)$ gives $w_{N,j} = \mathrm{O}(N^{-3/2})$. Thus each fixed coefficient tends to zero while $\Vert b_N^0\Vert = 1$. Boundedness and coordinatewise convergence imply $b_N^0\rightharpoonup0$ in $L^2(0,\pi)$. Finally, one has $\sum_{j \geq 1}j^{-2s}j^2<+\infty$ exactly when $s>3/2$.

Formula \eqref{normal_invariants} gives $\alpha_N(A_N) = -\sum_{j = 1}^Nj^2w_{N,j}$.
Using the fourth centered moment of $K_N$, one has
$\sum_{j = 1}^Nj^2w_{N,j} = \frac{2}{N}\mathbb E((K_N - N)^4) = \frac{3N - 1}{2}$.
Substitution in \eqref{normal_weight_displacement} gives the first expansion in \eqref{heat_weights}. Since $\frac{\tr A_N}{N} = -\frac{(N + 1)(2N + 1)}{6}$, the relative correction in \eqref{ratio_jet} is $(N - 1)(N - 2)/3$. The nondegenerate normal maximization in Proposition~\ref{prop_normal_weights} improves the general remainder to $\mathrm{O}(T^2)$ for fixed $N$, and \eqref{heat_ratio_jet} follows. No uniformity of this remainder in $N$ is claimed.
\end{proof}

The density \eqref{heat_bulk} makes the loss of compactness precise: most of the unit profile lies at modes of order $\sqrt N$. After renormalization by the first coefficient, the limiting moduli grow like $j$ and have the regularity of a first derivative of a point source. We call this a dipole signature. The signs have not been selected, so no convergence to one particular derivative of a Dirac mass is asserted. Joint limits in $N$ and $T$ require a spatial actuator class and uniform spectral estimates; see \cite{FattoriniRussell1971,LebeauRobbiano1995,Miller2004,TenenbaumTucsnak2007,PrivatTrelatZuazua2015,PrivatTrelatZuazua2017}.

\subsection{Schr\"odinger: from the small-time profile to revival}

Consider the complex truncation
\begin{equation}\label{schrodinger_truncation}
A_N = -i\diag(1^2,2^2,\ldots,N^2)
\end{equation}
with one complex input profile $b_N\in\C^N$. Let $G_T^\C(A_N,b_N) = \int_0^T e^{tA_N}b_Nb_N^*e^{tA_N^*}\,dt$ be the Hermitian Gramian, where $^*$ denotes the adjoint.
The covariance relaxation is taken over Hermitian positive semidefinite matrices. Since the flow is unitary, its value is $T/N$. The normal formulas and their proofs extend from transposes to adjoints. In particular, the two-term expansion \eqref{fixed_pair_jet} holds for complex controllable pairs and Hermitian Gramians, with $q_n^*Aq_n$ replaced by $\operatorname{Re}(q_n^*Aq_n)$; see \cite[Remark~1.5 (complex extension)]{TrelatZuazua2026gramian}. Here $A_N^* = -A_N$, so this real part vanishes. Hence, for every fixed $N$ and all sufficiently small $T>0$, the unique maximizing weight vector satisfies $\vert b_{N,j}(T)\vert^2 = w_{N,j} + \mathrm{O}(T^2)$, with $w_{N,j}$ as in \eqref{heat_weights}; the linear correction vanishes because the spectrum is purely imaginary, but this does not assert exact equality with $w_{N,j}$ at positive time. The phases are arbitrary. At special finite horizons the maximizing weights become completely flat.

\begin{proposition}[Schr\"odinger recovery horizons]
For \eqref{schrodinger_truncation}, scalar recovery is exact if and only if
\begin{equation}\label{schrodinger_phase_condition}
e^{-ij^2T} = e^{-ik^2T} \qquad\forall j,k\in\{1,\ldots,N\}.
\end{equation}
Equivalently, every $T>0$ is a recovery horizon when $N = 1$; the recovery horizons are $T\in(2\pi/3)\N^*$ when $N = 2$, and $T\in2\pi\N^*$ when $N \geq 3$.

At every recovery horizon, the maximizing unit profiles are exactly those satisfying
\begin{equation}\label{schrodinger_flat}
\vert b_{N,j}\vert^2 = \frac{1}{N} \qquad\forall j\in\{1,\ldots,N\},
\end{equation}
so their squared moduli are unique, while their phases are arbitrary, and $G_T^\C(A_N,b_N) = \frac{T}{N}\Id$.
After multiplication by $\sqrt N$, their modal moduli are constant. The corresponding formal profiles belong to $H_D^{-s}$ exactly when $s>1/2$.
\end{proposition}

\begin{proof}
In the modal basis, one has $(G_T^\C(A_N,b_N))_{jk} = b_{N,j}\overline{b_{N,k}}\int_0^T e^{-i(j^2 - k^2)t}\,dt$.
If this Gramian equals $T\,\Id/N$, its diagonal entries give \eqref{schrodinger_flat}, so no coefficient vanishes. Its off-diagonal entries vanish exactly under \eqref{schrodinger_phase_condition}. The converse follows from the same formula.

For $N = 2$, the only nonzero frequency difference has modulus three. For $N \geq 3$, the consecutive differences $2^2 - 1^2 = 3$ and $3^2 - 2^2 = 5$ are coprime. This gives the stated horizons. The regularity threshold follows from convergence of $\sum_{j \geq 1}j^{-2s}$.
\end{proof}

The same truncation therefore has two genuinely different signatures. Fast control gives the concentrated heat weights and the dipole threshold $s>3/2$, whereas a revival gives flat weights and the point-source threshold $s>1/2$. This is a change in the moduli of the finite-dimensional maximizers, not only in their phases. The statement concerns one complex input. The corresponding real scalar-input problem is different, since a complex input generally represents two real channels; its generator has the signed frequencies $\{\pm j^2\}$ and a different small-time discriminant.

\subsection{Wave: recovery under the physical forcing constraint}
For the wave equation, both the state metric and the admissible action of the control must be specified. Write the $j$th modal state in energy coordinates as $z_j = (jy_j,\dot y_j)\in\R^2$. For a spatial actuator $b_N = \sum_{j = 1}^Nb_{N,j}\varphi_j$, $\sum_{j = 1}^Nb_{N,j}^2 = 1$, define the induced state-space actuator $\widehat b_N = \bigoplus_{j = 1}^N(0,b_{N,j})^\top$. The first $N$ modes driven by the scalar control $u(t)$ read
\begin{equation}\label{wave_truncation}
\dot z = A_Nz + \widehat b_Nu(t),
\qquad
A_N = \bigoplus_{j = 1}^N(-jJ),
\qquad J = \begin{pmatrix}0&-1\\1&0\end{pmatrix}.
\end{equation}
Thus $A_N$ is skew-symmetric for the energy metric, but the actuator is restricted to the velocity direction of every modal plane.

\begin{proposition}[wave recovery at a common period]
Among the physically constrained actuators in \eqref{wave_truncation}, for $T = 2\pi M$ with $M\in\N^*$, one has
$$
\max_{\sum_jb_{N,j}^2 = 1}\lambda_{\min}(G_T(A_N,\widehat b_N)) = \frac{T}{2N}.
$$
The maximizing profiles are those satisfying $b_{N,j}^2 = \frac{1}{N}$ for every $j\in\{1,\ldots,N\}$; their squared coefficients are unique, but their $N$ signs are arbitrary. They attain the unrestricted covariance relaxation in dimension $2N$, despite the physical forcing constraint. After multiplication by $\sqrt N$, their modal moduli are constant and have the threshold $s>1/2$ in \eqref{HDs}.
\end{proposition}

\begin{proof}
Over an integer number of common periods, the functions $\cos(jt)$ and $\sin(jt)$ are mutually orthogonal for distinct positive integers $j$. The Gramian is block diagonal, and its $j$th diagonal block is
$$
b_{N,j}^2\int_0^T e^{-jtJ}\begin{pmatrix}0&0\\0&1\end{pmatrix}e^{jtJ}dt = \frac{Tb_{N,j}^2}{2}\,\Id_{\R^2}.
$$
Consequently,
$\lambda_{\min}(G_T(A_N,\widehat b_N)) = \frac{T}{2}\min_jb_{N,j}^2 \leq \frac{T}{2N}$,
with equality if and only if $b_{N,j}^2 = \frac{1}{N}$ for every $j$. Since the state dimension is $2N$, the unrestricted relaxed value \eqref{skew_relaxed} is also $T/(2N)$.
\end{proof}

The following comparison changes the state norm and therefore the worst-direction criterion; it is not merely a coordinate rewrite. Keep the same physical actuator class but measure the state in the unweighted coordinates $(y_j,\dot y_j)$. At $T = 2\pi M$, the $j$th diagonal block of the Gramian is then $\frac{Tb_{N,j}^2}{2}\diag(j^{-2},1)$.
The maximal value becomes $T/(2\sum_{j = 1}^Nj^2)$ and the maximizing weights become $b_{N,j}^2 = j^2/(\sum_{k = 1}^Nk^2)$.
The squared coefficients are again unique and their signs arbitrary, and division by the first modal coefficient gives $\vert b_{N,j}\vert/\vert b_{N,1}\vert = j$, hence the threshold $s>3/2$ in \eqref{HDs}. Thus the class of admissible physical forcings is unchanged, but its maximizing profile depends on the state metric: the energy metric produces a flat point-source signature, whereas the unweighted metric produces a dipole signature. Allowing arbitrary directions in every two-dimensional modal plane would define yet another actuator class. Passing to the full equation would additionally require an admissible control space, since the flat limiting profile is generally distributional. Related optimal observation and controller-location problems for one-dimensional waves, including modal approximations, are studied in \cite{PrivatTrelatZuazua2013observation,PrivatTrelatZuazua2013controllers}; see also \cite{Ingham1936,PrivatTrelatZuazua2016,TucsnakWeiss2009}.

\section{Conclusion and perspectives}
The rank of the actuator, the control horizon and the spectral geometry of the dynamics form a single design problem. Changing the horizon can alter not only optimal performance, but also the rank, spectral weights and geometry of the optimal actuator. At small times, the relaxed maximizer has full rank. When $A$ is cyclic, the fraction captured by a design of rank at most $k<n$ is equivalent to $n c_k(A)T^{2\lceil n/k\rceil - 2}$, with $c_k(A)>0$; at rank one, both the leading coefficient and its first correction are explicit. For symmetric dynamics, the full-rank obstruction persists at every positive horizon.

Skew-symmetric dynamics can instead admit exact low-rank recovery. At phase-aligned horizons, orbit averaging can make a low-rank covariance exactly isotropic and close the gap. Scalar recovery occurs precisely when the matrix is cyclic and the endpoint flow is a common sign, while the endpoint phase classes determine the least recovery rank in general. Away from equality, a sufficiently small defect forces a near return at the scale identified in Theorem~\ref{thm_rigidity}, whereas ordinary long-time cancellation can make the ratio tend to one without convergence to a fixed return.

The heat and Schr\"odinger truncations have the same concentrated small-time profile, whereas Schr\"odinger revivals and common wave periods produce flat modal weights. After renormalization, these regimes have respectively dipole-type and point-source-type envelopes. These results identify several distinct rank-time regimes governed by spectral geometry.

Open questions include the global rank-constrained optimizer and least recovery rank away from small time and phase alignment; exactness of the fixed-covariance lower bound for $c_k(A)$ in Proposition~\ref{prop_rank_gap} and accumulation of intermediate-rank optimizers; stability of phase classes; the joint truncation-horizon PDE limit; and removal of the factor $\sqrt n$ in Theorem~\ref{thm_rigidity} and Corollary~\ref{cor_gap}.

\appendix
\section{Auxiliary results and correction}

\subsection{Real covariance realization at the recovery-rank threshold}

The lower bound in Theorem~\ref{thm_phase_rank} follows from the orthogonality of distinct endpoint phase classes. The next construction shows that requiring a real covariance causes no additional loss of rank.

\begin{lemma}[real phase-block realization]\label{lem_real_realization}
Use the notation \eqref{spectral_decomposition}-\eqref{phase_multiplicity}. There exists a real covariance $X\succeq0$ satisfying $P_\omega XP_\omega = \frac{1}{n}P_\omega$, $P_\omega XP_\nu = 0$ if $e^{i\omega T}\neq e^{i\nu T}$, and $\displaystyle\rank X = \sum_{\zeta\in\sigma(e^{TA_\C})}m_\zeta(T)$.
\end{lemma}

\begin{proof}
For each endpoint phase $\zeta$, choose a complex Hilbert space $F_\zeta$ of dimension $m_\zeta(T)$. For every $\omega\in\Omega_\zeta(T)$, choose an isometry $Q_\omega:E_\omega\to F_\zeta$. These choices can be made compatible with complex conjugation. For a nonreal pair $\{\zeta,\overline\zeta\}$, choose the maps for one phase and define those for the other by conjugation. If $\zeta\in\{-1,1\}$, endow $F_\zeta$ with a conjugation, choose the maps for positive frequencies, define the negative-frequency maps by conjugation, and choose a real isometry on $E_0$ if the zero frequency belongs to the class.

On the phase block $\bigoplus_{\omega\in\Omega_\zeta(T)}E_\omega$, define $X_\zeta$ by $P_\omega X_\zeta P_\nu = \frac{1}{n}Q_\omega^*Q_\nu$.
This is a positive Gram operator. Its diagonal block on $E_\omega$ is $P_\omega/n$, and its rank is $m_\zeta(T)$ because one of the isometries is onto $F_\zeta$. Take the orthogonal sum over the endpoint phases. Compatibility with conjugation shows that the resulting operator is the complexification of a real positive semidefinite matrix. Its trace is one, and the ranks of the mutually orthogonal phase blocks add as claimed.
\end{proof}

\subsection{The error in the separated Brunovsk\'y objective}\label{sec_brunovsky}
Proposition~1 of \cite{GeshkovskiZuazua2022} contains an error, corrected in~\cite{GeshkovskiZuazuaCorrigendum}. In their notation, $P(b)$ is the Brunovsk\'y change of coordinates of Lemma~1, $\kappa(T)=\Vert\Gamma_{(\mathfrak A,e_n)}\Vert$ is the companion-pair null-control cost from Eq.~(15), and $\mathfrak C(b,T)=\mathcal C_T(A,b)$ is the worst null-control cost. The two displays after Eq.~(17) infer $\mathcal C_T(A,b)=\kappa(T)\Vert P(b)^{-1}\Vert$ from $\Vert P(b)^{-1}y_0\Vert\leq\Vert P(b)^{-1}\Vert\Vert y_0\Vert$. This is only an upper bound: since $\Gamma_b=\Gamma_{(\mathfrak A,e_n)}P(b)^{-1}$, the exact cost is $\Vert\Gamma_{(\mathfrak A,e_n)}P(b)^{-1}\Vert$, and equality in sub-multiplicativity need not hold.

For the exact relation, write $A=P\mathfrak A P^{-1}$ and $b=Pe_n$, where $\mathfrak A$ is the companion matrix of the characteristic polynomial of $A$, and set $R_T^-=G_T(-\mathfrak A,e_n)$. The identity $G_T(A,b)=e^{TA}G_T(-A,b)e^{TA^\top}$ turns the null-control cost for $y'=Ay+bu$ into the worst reachability cost for $-A$, while the change of coordinates gives $G_T(-A,b)=PR_T^-P^\top$.
Consequently, the exact squared worst null-control cost is
\begin{equation}\label{brunovsky_exact_cost}
\mathcal{C}_T(A,b)^2
= \lambda_{\max}((P^{-1})^\top (R_T^-)^{-1}P^{-1})
= \Vert (R_T^-)^{-1/2}P^{-1}\Vert^2.
\end{equation}
Sub-multiplicativity yields only $\mathcal{C}_T(A,b)^2 \leq \lambda_{\max}((R_T^-)^{-1})\Vert P^{-1}\Vert^2$.
Thus the factorization asserted in \cite[Proposition~1]{GeshkovskiZuazua2022} and its claimed equivalence with \cite[Eq.~(13)]{GeshkovskiZuazua2022} are false in general. Equation~(13) and the computations based on Eq.~(18) optimize the horizon-independent surrogate $\Vert P^{-1}\Vert$, so the $T$-independence claimed in \cite[Remark~4]{GeshkovskiZuazua2022} does not follow. The congruence itself remains exact, and the same nonseparability applies to the reachability cost $C_T(A,b)$ studied here, with $G_T(A,b)$ in place of $G_T(-A,b)$.

The following two-dimensional counterexample disproves the separated factorization and shows that its surrogate selects a different leading small-time actuator.
Take $A = \diag(1,2)$, $b = (x,y)^\top$ with $x^2 + y^2 = 1$ and $xy\neq0$.
With the convention of \cite[Proposition~1]{GeshkovskiZuazua2022},
$$
P = \begin{pmatrix}-2x&x\\-y&y\end{pmatrix}, \qquad \mathfrak{A} = \begin{pmatrix}0&1\\-2&3\end{pmatrix}, \qquad A = P\mathfrak{A} P^{-1}, \qquad b = Pe_2.
$$
Direct calculation gives
$$
P^{-1} = \begin{pmatrix}-1/x&1/y\\-1/x&2/y\end{pmatrix}, \qquad \det(13\,\Id - (P^{-1})^\top P^{-1}) = -\frac{(13x^2 - 5)^2}{x^2y^2} \leq 0.
$$
Thus $\Vert P^{-1}\Vert^2 \geq 13$, with equality at $x^2 = 5/13$. At this point the eigenvalues of $(P^{-1})^\top P^{-1}$ are $13$ and $13/40$. The horizon-independent surrogate in \cite[Eq.~(13)]{GeshkovskiZuazua2022} therefore selects $x^2 = 5/13$.

The true small-time null-control objective selects another direction. Here $d_2(-A,\allowbreak b) = \vert xy\vert$ and $\kappa_2 = 1/12$, so the fixed-pair expansion \eqref{fixed_pair_jet}, applied to $(-A,b)$, gives
$\mathcal{C}_T(A,b)^2 = \frac{12}{x^2(1 - x^2)T^3}(1 + \mathrm{O}(T))$.
Its leading coefficient is uniquely minimized at $x^2 = 1/2$. At this point,
$\Vert P^{-1}\Vert^2 = 7 + 3\sqrt5$ and $\lambda_{\max}((R_T^-)^{-1}) = 12T^{-3}(1 + \mathrm{O}(T))$,
whereas the exact squared cost is $48T^{-3}(1 + \mathrm{O}(T))$. The ratio between the separated upper bound and the exact squared cost is therefore $\frac{7 + 3\sqrt5}{4} + \mathrm{O}(T)$.

\begingroup
\section*{Acknowledgements}
The second author was partially supported by the European Research Council (ERC) through the European Union's Horizon Europe programme (ERC Advanced Grant CoDeFeL, grant agreement No.~101096251); the Air Force Office of Scientific Research under award No.~FA8655-24-1-7027; the Alexander von Humboldt Professorship; the European Union's Horizon Europe MSCA Doctoral Network ModConFlex (grant agreement No.~101073558); the Research Council of Norway through SURE-AI, grant No.~357482; and Grant PID2023-146872OB-I00 (DyCMaMod), funded by MICIU/AEI/10.13039/501100011033 and by ERDF/EU\@. This article is based upon work from COST Actions CA24122 (mSPACE) and CA24136 (InterCoML), supported by COST (European Cooperation in Science and Technology). 
\par\smallskip
\noindent
AI tools (ChatGPT and Claude) were used solely to improve the exposition, check the bibliography, and compare versions. All scientific ideas, results, proofs, and initial drafts are the authors' own. The authors assume responsibility for all content.
\par
\endgroup

\end{document}